\documentclass[reqno]{article}
\usepackage{amsmath,amsfonts,amssymb,amsthm}

\newtheorem{theorem}{Theorem}[section]
\newtheorem{lemma}[theorem]{Lemma}

\newtheorem{corollary}[theorem]{Corollary}

\theoremstyle{definition}
\newtheorem{remark}[theorem]{Remark}

\newtheorem*{ack}{Acknowledgement}

\newenvironment{romenumerate}{\begin{enumerate}
 }{\end{enumerate}}

\newcommand\noproof{\hfill$\Box$}

\newcommand\Bi{{\mathrm{Bi}}}

\newcommand\eps{\varepsilon}
\newcommand\la{\lambda}

\renewcommand\Pr{{\mathbb P}}
\newcommand\E{{\mathbb E}}
\newcommand\Var{\operatorname{\mathrm{Var}}}
\newcommand\Cov{\operatorname{\mathrm{Cov}}}
\newcommand\dto{\overset{\mathrm{d}}{\to}}

\newcommand\tX{{\widetilde X}}

\newcommand\op{o_{\mathrm{p}}}
\newcommand\Op{O_{\mathrm{p}}}
\newcommand\cF{\mathcal{F}}
\newcommand\cA{\mathcal{A}}
\newcommand\cB{\mathcal{B}}
\newcommand\cC{\mathcal{C}}
\newcommand\cE{\mathcal{E}}
\newcommand\dmax{d_{\mathrm{max}}}
\newcommand\bb[1]{\bigl(#1\bigr)}
\newcommand\vd{{\bf d}}
\newcommand\mm{{\mathrm{m}}}
\newcommand\dd{\mathrm{d}}
\newcommand\pP{\mathcal{P}}
\newcommand\xdot{\dot{x}}
\newcommand\tmax{t_{\mathrm{max}}}
\newcommand\bigabs[1]{\bigl|#1\bigr|}
\newcommand\floor[1]{\lfloor #1 \rfloor}

\newcommand\Bern{\mathrm{Bern}}
\newcommand\bI{{\bf I}}
\newcommand\bW{{\bf W}}
\newcommand\La{\Lambda}
\newcommand\dtwo{c_0}
\newcommand\alphas{{\alpha_0,\alpha_1,\alpha_2}}
\newcommand\R{\mathbb{R}}

\newcommand\newbe{\theta}

\begin{document}
\title{The phase transition in the configuration model}
\author{Oliver Riordan%
\thanks{Mathematical Institute, University of Oxford, 24--29 St Giles', Oxford OX1 3LB, UK and
Department of Mathematical Sciences, University of Memphis, Memphis TN 38152, USA.
E-mail: {\tt riordan@maths.ox.ac.uk}.}}
\date{April 4, 2011}
\maketitle

\begin{abstract}
Let $G=G(\vd)$ be a random graph with a given degree sequence $\vd$, such as a random
$r$-regular graph where $r\ge 3$ is fixed and $n=|G|\to\infty$.
We study the percolation phase transition on such graphs $G$,
i.e., the emergence as $p$ increases of a unique giant component
in the random subgraph $G[p]$
obtained by keeping edges independently with probability $p$.
More generally, we study the emergence
of a giant component in $G(\vd)$ itself as $\vd$ varies.
We show that a single method can be used to prove very
precise results below, inside and above the `scaling window'
of the phase transition,
matching many of the known results for the much simpler model $G(n,p)$.
This method is a natural extension of that used by Bollob\'as
and the author to study $G(n,p)$, itself
based on work of Aldous and of Nachmias and Peres; 
the calculations are significantly more involved in the present setting.
\end{abstract}

\section{Introduction and results}\label{sec_intro}

In 1997, Aldous showed that inside the `scaling window' of the phase
transition, i.e., when $p=(1+\alpha n^{-1/3})/n$ with $\alpha=O(1)$,
the rescaled sizes of the largest components of $G(n,p)$ converge
to a certain distribution related to Brownian motion.
His proof was based on a natural exploration process, introduced
in the context of random graphs by Karp~\cite{Karp}
and considered in a closely related context a little
earlier by Martin-L\"of~\cite{ML86}, but with a twist: after
finishing exploring a component one starts exploring the next component
in such a way that a certain quantity related to the exploration
behaves very much like a random walk with independent increments.

Recently, Nachmias and Peres~\cite{NP_giant} used the same
process (with the same `restarts') to study $G(n,p)$ \emph{outside}
the scaling window, giving a simpler proof of somewhat weaker
forms of known results for this case; in~\cite{BR_walk}
Bollob\'as and the author showed that with a little more work,
much stronger results
could be proved by analyzing the same exploration process in a new way.

Nachmias and Peres~\cite{NP_regular} adapted
their approach to study the phase transition in random $r$-regular
graphs, i.e., the emergence of a giant component in the subgraph
of a random $r$-regular graph obtained by selecting edges
with probability $p$; they showed in particular that the
`window' is when $p=(1+\Theta(n^{-1/3}))/(r-1)$.  Independently,
Janson and Luczak~\cite{JL_new} used a different approach to prove
the supercritical part of this result in a more general context.

In this paper we shall extend and generalize the results mentioned
in the previous paragraph.
Firstly, rather than a random subgraph of a random regular graph,
we study the `configuration model' of Bollob\'as~\cite{BB_config}, 
giving a random (simple or multi-)graph with a given (here bounded) degree sequence;
it is easy to see that random subgraphs of random regular
graphs can be viewed in this way (see Fountoulakis~\cite{F_percd}). Secondly, we prove
much more precise results, obtaining essentially the full
strength of the corresponding results for $G(n,p)$.
In particular, we show that above the window the size
of the giant component is asymptotically normally distributed,
corresponding to the result of Pittel and Wormald~\cite{PWio}
for $G(n,p)$, and prove results equivalent to those of Aldous~\cite{Aldous}
in the critical case. The approach used here is very much that
of Bollob\'as and the author in~\cite{BR_walk}, adapted to the configuration
model.

Throughout we assume that we are `near' the phase transition, in the
range corresponding to $np\to 1$ in $G(n,p)$. Perhaps surprisingly,
the proofs are easier the closer the graphs are to critical.  The
method used here could probably be extended to the more strongly
supercritical case (covered for $G(n,p)$ in~\cite{BR_walk}), but
further away from the phase transition other approaches (based on
studying small components) are likely to be simpler.  We assume here
that the maximum degree remains bounded; this assumption can doubtless
be weakened.

Before turning to the details, let us comment briefly on the history
of this problem.  The existence and size of the giant component in the
configuration model were first studied by Molloy and
Reed~\cite{MolloyReed1,MolloyReed2}, who found
the size of the largest component up to a $o(n)$ error (not only near
the phase transition, of course).  It is
easy to check that a random \emph{subgraph} of a random graph generated
by the configuration model is again an instance of the configuration
model.  Hence, studying the percolation phase transition in the
(random) environment of the configuration model reduces to studying
the transition in the configuration model itself as its parameters are
varied. (This is spelled out in detail by Fountoulakis~\cite{F_percd};
see also Janson~\cite{Janson_cperc}.)
Nevertheless, percolation on random $r$-regular graphs has
received separate attention; for many proof techniques this special case is much easier
to handle. In this special case, the critical point of the phase
transition was established explicitly by Goerdt~\cite{Goerdt},
and Benjamini raised the question
of finding the `window' of the phase transition (see~\cite{Pittel_rrg,NP_regular}).
Results establishing the approximate width (either for this special case
or more generally for the configuration model itself)
were given recently by Kang and Seierstad~\cite{KS},
Pittel~\cite{Pittel_rrg} and Janson and Luczak~\cite{JL_new}, in all
cases with logarithmic gaps in the bounds. The exact width of the window,
and the asymptotic size of the largest component in all ranges, was found
by Nachmias and Peres~\cite{NP_regular}, but only for the case of random
subgraphs of random $r$-regular graphs. As mentioned above, here we not only
extend this result to the configuration model, but greatly improve the precision,
establishing not only the asymptotic size of the largest component
above and below the window, but also the scale and limiting distribution
of its fluctuations.

\medskip
Let $\vd=\vd_n=(d_1^{(n)},\ldots,d_n^{(n)})$ be a \emph{degree sequence}, i.e., a sequence
of non-negative integers with even sum.
For the moment we assume only that all degrees are at most
some constant $\dmax\ge 2$. In the following results the sequence of course depends on
$n$, but often we suppress this in the notation, writing $d_i$ for the degree of vertex $i$,
for example.

Let $G_{\vd}^\mm$ be the \emph{configuration multigraph} with degree sequence $\vd$,
introduced by Bollob\'as~\cite{BB_config},
defined as follows.
Let $S_1,\ldots,S_n$ be disjoint sets with $|S_i|=d_i$; we call
the elements of the $S_i$ \emph{stubs}. Let $\pP$ be a pairing
of $\bigcup S_i$, i.e., a partition of $\bigcup S_i$ into parts of size 2, chosen
uniformly at random from all such pairings. Form the multigraph $G_{\vd}^\mm$ from $\pP$
by replacing each pair $\{s,t\}$ with $s\in S_i$ and $t\in S_j$ by an edge with endvertices $i$ and $j$.

Let us write $L_i(G)$ for the number of vertices in the $i$th largest (when sorted by number of vertices)
component of a graph $G$, noting that the definition is unambiguous even if there
are ties in the component sizes.
Our main aim is to study the distribution of $L_1(G_{\vd}^\mm)$.

When it comes to asymptotic results, we shall always make the following two assumptions.
Firstly, there is a constant $\dmax$ such that all degrees satisfy
\begin{equation}\label{a1}
 d_i^{(n)} \le \dmax.
\end{equation}
Secondly, there is a constant $\dtwo>0$ such that
\begin{equation}\label{a2}
 |\{i:d_i^{(n)} \notin \{0,2\}\}|\ge \dtwo n
\end{equation}
for all large enough $n$.
There are two aspects to this second condition. Firstly, degree-$0$
vertices play no role in the construction, so we could just as 
well rule them out. However, in applications we shall consider
sequences containing some zeros, and to avoid an extra rescaling step
(where $n$ is replaced by the number of non-zero degrees), it
is convenient to allow them.

The situation with degree-2 vertices is somewhat similar: apart from
the possibility of cycle components, the multigraph $G_{\vd}^\mm$ is a
random subdivision of the multigraph $G_{\vd'}^\mm$, where $\vd'$ is
obtained by deleting all degree-2 vertices from $\vd$. For the scaling
behaviour, it is the number of (non-isolated) vertices in $\vd'$ that
matters, not the number in $\vd$. The condition \eqref{a2} ensures
that $n$ is (up to a constant) the correct scaling parameter
in conditions such as $\eps^3n\to\infty$ below.

Given a degree sequence $\vd$, let
\begin{equation}\label{mudef}
 \mu_r =\mu_r(\vd) = n^{-1}\sum_{i=1}^n (d_i)_r
\end{equation}
denote the $r$th \emph{factorial moment} of $\vd$,
where $(x)_r=x(x-1)\cdots(x-r+1)$.
Let
\begin{equation}\label{ladef}
 \la = \la(\vd) = \frac{\sum d_i(d_i-1)}{\sum d_i} = \frac{\mu_2}{\mu_1}.
\end{equation}
As is well known (see, e.g., \cite{MolloyReed1,rg_bp}),
the quantity $\la$ corresponds to the average `branching factor' in $G_{\vd}^\mm$.
Our main interest is the `weakly supercritical' case where $\la\to 1$ from above,
but $(\la-1)n^{1/3}\to\infty$, so we are outside the `scaling window'
of the phase transition. However, we shall also prove results for the critical
and weakly subcritical cases.

We usually write $\la=\la(n)$ as $1+\eps(n)$, and assume throughout that
\begin{equation}\label{a3}
 \la=\frac{\mu_2}{\mu_1}\to 1
\end{equation}
as $n\to\infty$, i.e., that $\eps\to 0$.

Let $\eta=\eta(n)$ denote the random variable obtained by choosing an element
of $\vd$ at random, with each element chosen with probability proportional to its
value. Thus $\la=\E(\eta-1)$ and $\eps=\E(\eta-2)$.
The quantity
\begin{equation}\label{v0def}
 v_0 = \Var(\eta)=\Var(\eta-2)
\end{equation}
will play an important role in our results. An elementary calculation shows that
\[
 v_0 = \frac{\mu_3\mu_1+\mu_2\mu_1-\mu_2^2}{\mu_1^2}.
\]
Under our assumptions we have $\E(\eta-2)\to 0$ and, from \eqref{a2}, $\sup\Pr(\eta=2)<1$,
which implies that $\Var(\eta-2)=\Theta(1)$. Since $\mu_2/\mu_1\to 1$, it follows that
\begin{equation}\label{vo1}
 v_0 \sim \frac{\mu_3}{\mu_1} =\Theta(1)
\end{equation}
as $n\to\infty$.

Writing $p_d=p_d(n)=n^{-1}|\{i:d_i=d\}|$ for the fraction of vertices with degree $d$,
for $\la>1$, let $z$ be the smallest solution in $[0,1]$ to the equation
\begin{equation}\label{prerho}
 z = \sum_d z^{d-1} \frac{d p_d}{\mu_1},
\end{equation}
and set
\begin{equation}\label{rhodef}
 \rho=\rho(\vd)=1- \sum_d z^d p_d.
\end{equation}
This quantity is most naturally seen as the survival probability
of a certain branching process (see \cite{rg_bp}; $z$ is essentially the probability
that when we follow a random edge, we end up in a small component).
What we call $\rho$ here is exactly the 
quantity $\epsilon_{\cal D}$ appearing in the result of Molloy and Reed~\cite{MolloyReed2}.
Let
\begin{equation}\label{rho*}
 \rho^*=\rho^*(\vd) = \sum_d z^d p_d -\frac{\mu_1 z^2}{2} -1 + \frac{\mu_1}{2};
\end{equation}
it will turn out that $\rho^*n$ will give asymptotically the nullity (or excess) of the
giant component. Although we shall not prove this, this is also asymptotically
the number of vertices in the `kernel'.

We shall show later (see Lemma~\ref{traj}\ref{v} and \eqref{rho*sim}) that
\begin{equation}\label{rhosim}
 \rho\sim\frac{2\mu_1^2}{\mu_3}\eps \text{\ \ and\ \ }\rho^*\sim \frac{2\mu_1^3}{3\mu_3^2}\eps^3.
\end{equation}
Molloy and Reed~\cite{MolloyReed2} showed, under different
conditions, that $L_1(G_{\vd}^\mm)=\rho(\vd) n+\op(n)$. 
(Their result allowed much larger degrees, and they did not assume that $\eps\to 0$;
indeed, their result does not `bite' in this case.)
Our main result concerns the fluctuations of $L_1(G_{\vd}^\mm)$ around $\rho(\vd)n$
in the weakly supercritical case.
\begin{theorem}\label{th1}
Let $\dmax\ge 2$ and $\dtwo>0$ be fixed. For each $n$ let $\vd=\vd_n$ be a degree sequence
satisfying \eqref{a1} and \eqref{a2}. Define $\mu_i$, $\la$, $\rho$ and $\rho^*$ as above,
noting that these quantities depend on $n$.
Setting $\eps=\la-1$, suppose that $\eps\to 0$ and $\eps^3 n\to\infty$.
Let $L_1$ and $N_1$ denote the order and nullity of the largest
component of $G_{\vd}^\mm$. Then $L_1'=L_1-\rho n$ and $N_1'=N_1-\rho^* n$
are asymptotically jointly normally distributed with mean 0,
\[
 \Var(L_1')\sim 2\mu_1\eps^{-1}n,\text{\ \ }\Var(N_1')\sim 5\rho^*n\sim \frac{10\mu_1^3}{3\mu_3^2}\eps^3n,
\text{\ and\ }\Cov(L_1',N_1')\sim \frac{2\mu_1^2}{\mu_3}\eps n.
\]
Furthermore,
\begin{equation}\label{12}
 L_2(G_{\vd}^\mm) = \Op\bb{\eps^{-2}\log(\eps ^3n)}.
\end{equation}
\end{theorem}
Here, as usual, given a sequence $(Z_n)$ of random variables 
and a deterministic function $f(n)$,
$Z_n=\Op(f(n))$ means that $Z_n/f(n)$ is bounded in probability, i.e.,
for any $\delta>0$ there exists $C$ such that $\Pr(|Z_n|\le Cf(n))\ge 1-\delta$
for all (large enough) $n$. We say that an event (formally
a sequence of events) holds \emph{with high probability} or \emph{whp}
if its probability tends to $1$ as $n\to\infty$. We write $Z_n=\op(f(n))$ if $Z_n/f(n)$ converges to $0$
in probability, i.e., if for any $\delta>0$ we have $|Z_n|\le \delta f(n)$ whp.

We assume a bounded maximum degree for simplicity. The proof extends
to the case where the maximum degree grows reasonably slowly; we have
not investigated this further.

The asymptotic correlation coefficient $\Cov(L_1',N_1')/\sqrt{\Var(L_1')\Var(N_1')}$
given by Theorem~\ref{th1}
is simply $\sqrt{3/5}$. Although this case is not covered by our result,
if we take the degree distribution to be Poisson as in $G(n,p)$ then
when $p\sim 1/n$ we have $\mu_i\sim 1$
for all $i$, and the variance and covariance formulae above are consistent
with those given by Pittel and Wormald~\cite[Note 4]{PWio} for $G(n,p)$.

The proof of Theorem~\ref{th1} will show that for each $d$ the number $L_1(d)$
of degree-$d$ vertices in the largest component satisfies
\begin{equation}\label{Ld}
 L_1(d) = n_d(1-z^d) +\Op(\sqrt{n/\eps}),
\end{equation}
where $n_d=np_d$
is the total number of degree-$d$ vertices and $z$ is defined by \eqref{prerho}.
The parameter $z$ corresponds to $\xi$ in \cite{JL_new},
so \eqref{Ld} refines the results there. Our method would allow us to establish
joint normality of these numbers with variances and covariances of order $n/\eps$,
but we shall not give the details.

We next consider the subcritical case.
Writing, as before,
$p_d=p_d(n)$ for the proportion of vertices with degree $d$, let $q_d=dp_d/\mu_1=\Pr(\eta=d)$
be the corresponding size-biased distribution.
Let $a_n$ be the (unique -- see Section~\ref{sec_below}) solution to
$\sum (d-2) q_d e^{a_n(d-2)}=0$, and define
\begin{equation}\label{dldef}
 \delta_n = -\log\left(\sum_d q_d e^{a_n(d-2)}\right).
\end{equation}
\begin{theorem}\label{th-1}
Let $\dmax\ge 2$ and $\dtwo>0$ be fixed. For each $n$ let $\vd=\vd_n$ be a degree sequence
satisfying \eqref{a1} and \eqref{a2}. Define $\mu_i$ and $\la$
as above, noting that these quantities depend on $n$.
Setting $\eps=1-\la$, suppose that $\eps\to 0$ and $\eps^3 n\to\infty$.
Then for all $x \in \R$ we have 
\begin{equation}\label{L1-+}
 \Pr\left( L_1(G_{\vd}^\mm) \le \delta_n^{-1}(\log\La -\frac{5}{2}\log\log\La +x)\right) = \exp\bb{-ce^{-x}} +o(1),
\end{equation}
where $\La=\eps^3 n$, $c=c(\vd)=\Theta(1)$ is given by \eqref{cform}, and $\delta_n$, defined in \eqref{dldef}, satisfies
\begin{equation}\label{dlsim}
 \delta_n \sim \frac{\eps^2}{2v_0}\sim \frac{\eps^2 \mu_1}{2\mu_3}.
\end{equation}
In particular,
\begin{equation}\label{L1-}
 L_1(G_{\vd}^\mm) = \delta_n^{-1}(\log\La -\frac{5}{2}\log\log\La +\Op(1)),
\end{equation}
and $L_1(G_{\vd}^\mm) = (2\mu_3/\mu_1+\op(1))\eps^{-2}\log \La$.
\end{theorem}

The bound \eqref{12} in Theorem~\ref{th1} is proved by applying
Theorem~\ref{th-1} to what remains of the supercritical graph
after deleting the largest component, so in Theorem~\ref{th1}
we in fact obtain bounds of the type \eqref{L1-+},\eqref{L1-}
but with $c$, $\delta_n$ and $\La$ defined for the `dual' distribution
with $n_dz^d$ vertices of each degree $d$; see~\eqref{Ld}
and the remark at the start of Subsection~\ref{ss_ext}.

Theorem~\ref{th-1} is the equivalent of (the corrected form of, see~\cite{rg_bp}) \L uczak's
extension~\cite{Luczak_near} of Bollob\'as's result~\cite{BB_evol}
for $G(n,p)$ in the subcritical case.

Finally, in the critical case, we obtain an analogue of the results of Aldous~\cite{Aldous}
for $G(n,p)$. The statement requires a few definitions, analogous to those in~\cite{Aldous}.

Let $(W(s))_{0\le s<\infty}$ be a standard Brownian motion.
Given real numbers $\alpha_0$, $\alpha_1$ and $\alpha_2$ with $\alpha_0>0$, let
\begin{equation}\label{Wdef}
 W_\alphas(s) = \alpha_0^{1/2} W(s) + \alpha_1 s -\frac{\alpha_2 s^2}{2}
\end{equation}
be a rescaled Brownian motion with drift $\alpha_1-\alpha_2 s$ at time $s$, and
set
\begin{equation}\label{Bdef}
 B(s) = B_\alphas(s) = W_\alphas(s) - \min_{0\le s'\le s}W_\alphas(s').
\end{equation}
Also, define a process $N(s)$ of `marks' so that, given $B(s)$, $N(s)$
is a Poisson process with intensity $\beta B(s)\dd s$,
where $\beta>0$ is constant. (For the formal details, see~\cite{Aldous}.)
Finally, order the excursions $\gamma_j$ of $B(s)$, i.e., the maximum intervals on which $B(s)$ is strictly
positive, in decreasing order of their lengths $|\gamma_j|$. Writing $N(\gamma_j)$ for the number
of marks in $\gamma_j$, this defines a joint distribution 
\begin{equation}\label{dist}
 \bb{ |\gamma_j|,N(\gamma_j) }_{j\ge 1}
\end{equation}
that depends on the parameters $\alpha_0$, $\alpha_1$, $\alpha_2$ and $\beta$.
\begin{theorem}\label{th0}
Suppose that $\vd=\vd_n$ satisfies assumptions \eqref{a1} and \eqref{a2} above.
Suppose also that $n^{1/3}(\la-1)$ converges to some $\alpha_1\in \R$,
that $\mu_3/\mu_1\to\alpha_0$, that $\mu_3/\mu_1^2\to \alpha_2$, and that $1/\mu_1\to\beta$.
Ordering the components by number of vertices, for any fixed $r$
the sequence $(n^{-2/3}|\cC_j|,n(\cC_j))_{j=1}^r$ converges in distribution
to the first $r$ terms of the sequence $(|\gamma_j|,N(\gamma_j))$ defined above,
where $n(\cC_j)$ is the nullity of $\cC_j$.
\end{theorem}

Note that in the light of \eqref{vo1}, assuming $\mu_3/\mu_1\to\alpha_0$
is equivalent to assuming $v_0\to\alpha_0$.

We have written Theorem~\ref{th0} with the scaling that arises most naturally
in the proof. From the scaling properties of Brownian motion, one can check
that $B_\alphas(s)$ is equal in distribution (as a process) to
$\alpha' B_{1,\alpha'_1,1}(\alpha'_0 s)$ where
$\alpha'=\alpha_0^{2/3}\alpha_2^{-1/3}$, $\alpha'_0=\alpha_2^{2/3}\alpha_0^{-1/3}$
and $\alpha'_1=\alpha_1\alpha_0^{-1/3}\alpha_2^{-1/3}$.
Hence, if we consider only the component sizes, there is a single-parameter
family of limiting processes, characterized
by $\alpha'_1$, the limiting value of $n^{1/3}(\la-1)\mu_1\mu_3^{-2/3}$.
Noting that the excursion lengths are scaled by $\alpha'_0$, Theorem~\ref{th0}
shows that if $n^{1/3}(\la-1)\mu_1\mu_3^{-2/3}\to\alpha$, then
$(n^{-2/3}\mu_3^{1/3}\mu_1^{-1} L_i)_{i=1}^r$ converges
to the first $r$ sorted excursion lengths of $B_{1,\alpha,1}$. These
excursion lengths are exactly the rescaled component sizes appearing
in Aldous's result for $G(n,1/n+\alpha n^{-4/3})$.

If we also consider the nullities, or mark counts, then there is a two-parameter
family of possible limits.

\subsection{Applications and Extensions}\label{ss_ext}

Let $\vd$ satisfy the assumptions above (in particular,
\eqref{a1}, \eqref{a2}, \eqref{a3} and $\La=\eps^3 n\to\infty$).
Observing at all times the restriction that all degrees are at most
$\dmax$, changing $m$ entries of $\vd$ changes the proportion $p_d$
of degree-$d$ vertices by $O(m/n)$,
and hence changes the quantities $\mu_1$, $\mu_2$ and $\mu_3$,
which are all of order $1$, by $O(m/n)$.
It follows that $\eps=\la-1=\mu_2/\mu_1-1$ changes by an absolute
amount that is $O(m/n)$, so if $m=o(\eps n)$ the relative change in $\eps$
is $O(m/(\eps n))$. It is easy to see, and will follow from the results
later in the paper, that the relative change in $\rho$, $\rho^*$,
$\La=\eps^3n$ or $\delta_n$ is of this same order $O(m/(\eps n))$,
as one would expect from the formulae~\eqref{rhosim} and \eqref{dlsim}.

When $m=o(\sqrt{n/\eps})$, the changes in $\rho n$ and $\rho^*n$ are small compared
to the relevant standard deviations, and
it follows that Theorem~\ref{th1} applies just as well to $G_{\vd'}^\mm$
for any $\vd_n'$ obtained from $\vd_n$ by changing 
$o(\sqrt{n/\eps})$ entries, even when all quantities in the conclusion
of the theorem are calculated for $\vd$ rather than for $\vd'$.
Similarly, if $m=O(\sqrt{n/\eps})$ then the relative
change in $\delta_n$ is $O(1/\sqrt{\La})=o(1/\log\La)$,
and it is not hard to check that Theorem~\ref{th-1} applies
to $G_{\vd'}^\mm$ in this case; for Theorem~\ref{th0},
the corresponding condition is simply $m=o(n^{2/3})$,
so that the limit of $n^{1/3}(\la-1)$ is unchanged.
Note that $m=O(\sqrt{n})$ satisfies the bound on $m$ in all
three cases.

One application of these
observations concerns random subgraphs of random $r$-regular graphs,
or indeed random subgraphs of configuration (multi-)graphs.

Let $\vd$ be the (random) degree sequence of the graph $G_r^\mm[p]$ obtained by starting
with an $r$-regular configuration multigraph, selecting
edges independently with probability $p$, and retaining
the selected edges. Conditional on $\vd$, the distribution
of $G_r^\mm[p]$ is simply that of $G_{\vd}$. Hence,
as noted by Fountoulakis~\cite{F_percd}, one can study
$G_r^\mm[p]$ by studying $G_{\vd}^\mm$.

It is very easy to check that for $0\le d\le r$,
the proportion $p_d$ of degree-$d$ vertices in $\vd$
satisfies 
\[ 
 p_d =p_d^0 +\Op(n^{-1/2}) \text{\ \ \ where\ } p_d^0 = \binom{r}{d} p^d(1-p)^{r-d}.
\]
It follows that the quantities $\mu_i$ defined above
satisfy
\[
 \mu_i = \mu_i^0 + \Op(n^{-1/2}) \text{\ \ \ where\ }\mu_i^0 =p^i(r)_i=p^ir(r-1)\cdots(r-i+1).
\]
Note that $\la^0=\mu_2^0/\mu_1^0=p(r-1)$, and when $\la^0\to 1$, i.e., $p\sim 1/(r-1)$,
then $\mu_1^0\sim r/(r-1)$ and $\mu_3^0\sim r(r-2)/(r-1)^2$.
From the remarks above, even though the actual number of vertices of each degree
is random, $G_{\vd}^\mm$ will satisfy the conclusions of Theorem~\ref{th1}--\ref{th0}
for the idealized sequence with $p_d$ replaced by $p_d^0$.
Hence, defining $\rho^0$ by \eqref{prerho} and \eqref{rhodef} with $p_d$ replaced by $p_d^0$,
Theorems~\ref{th1}--\ref{th0} have the following consequence.
(We omit the nullity result and the analogue of \eqref{L1-+} for simplicity; these also carry over.)

\begin{corollary}\label{crreg}
Let $r\ge 3$ be fixed and let $p=(1+\eps)/(r-1)$ where $\eps=\eps(n)\to 0$.
Let $G$ be the random subgraph of the random $r$-regular configuration multigraph
on $n$ vertices obtained by selecting
edges independently with probability $p$.

If $\eps>0$ and $\eps^3n\to\infty$ then
\[
 \frac{L_1(G) - \rho^0 n}{\sigma\sqrt{n}} \dto N(0,1),
\]
where $\rho^0$ (defined above) satisfies $\rho^0\sim 2\eps r/(r-2)$, and $\sigma^2= 2\eps^{-1}r/(r-1)$.

If $\eps<0$ and $|\eps|^3n\to\infty$ then
\[
 L_2(G) = \delta_n^{-1}(\log\La -\frac{5}{2}\log\log\La +\Op(1)),
\]
where $\La=|\eps|^3 n$ and $\delta_n$, defined by \eqref{dldef} with $p_d^0$ in place
of $p_d$, satisfies
$\delta_n\sim \frac{r-1}{2(r-2)}\eps^2$.

Finally, if $n^{1/3}\eps\to\alpha_1\in\R$ then the sizes and nullities
of the components sorted in decreasing order of size
converge in distribution to the distribution described in \eqref{dist} with $\alpha_0=(r-2)/(r-1)$,
$\alpha_2=(r-2)/r$ and $\beta=(r-1)/r$.\noproof
\end{corollary}
Note that this result is consistent with, but much sharper than, the results
of Nachmias and Peres~\cite{NP_regular}.   Of course, one can formulate
a similar corollary concerning the random subgraph $G_{\vd}^\mm[p]$
of $G_{\vd}^\mm$, for any degree sequence $\vd$ satisfying \eqref{a1} and \eqref{a2};
passing to the subgraph multiplies $\mu_i$ by $p^i$, just
as in the $r$-regular case, and $p$ must be chosen so that the value
of the `branching factor' $\la$ in the subgraph ($p$ times
that in the original) is of the form $1+\eps$ with $\eps\to 0$.

As shown by Bollob\'as~\cite{BB_config} (see also~\cite{BC}), when the maximum degree
is bounded, the probability that the configuration multigraph $G_{\vd}^\mm$ is 
simple is bounded away from 0, and conditional on this event,
$G_{\vd}^\mm$
has the distribution of $G_{\vd}$, a uniformly random simple graph with degree sequence $\vd$.
It follows that any `whp' results for $G_{\vd}^\mm$ transfer to $G_\vd$. This applies
to Theorem~\ref{th-1}, in the weaker form \eqref{L1-},
but not to the other results above. More precisely,
as shown in~\cite{BB_config} (or, after translating from the enumerative
to probabilistic viewpoint, \cite{BC}),
when the maximum degree is bounded,
the probability $p$ that $G_{\vd}^\mm$ is simple satisfies
$p\sim\exp(-\theta-\theta^2)$, where $\theta=\sum_i \binom{d_i}{2}/\sum_i d_i$,
which in our notation is $\mu_2/(2\mu_1)$.

The proofs of Theorems~\ref{th1}--\ref{th0} involve `exploring' part of the graph.
We can end these explorations when certain entire components have been revealed,
comprising in total $o(n)$ vertices. It is easy to check that the probability
of encountering a loop or multiple edge in the exploration is $o(1)$.
Moreover, the unexplored part of the graph may be seen as a configuration
multigraph $G'$. Since only $o(n)$ vertices have been explored, the (conditional)
probability that $G'$ is simple is $p+o(1)$: the corresponding $\theta$
is within $o(1)$ of the original $\theta$. It follows that if $\mathcal{E}$ is some
not too unlikely event
defined in terms of our exploration, then the probability that $\mathcal{E}$ holds
and $G_{\vd}^\mm$ is simple is $(\Pr(\mathcal{E})+o(1))(p+o(1)) \sim p \Pr(\mathcal{E})$.
Thus conditioning on $G_{\vd}^\mm$ being simple hardly changes the probability
of $\mathcal{E}$. Using this observation it is easy to transfer the results
above to random simple graphs; we omit the details.
\begin{theorem}
All the results above apply unchanged if $G_\vd^\mm$ is replaced by
the random simple graph $G_\vd$.\noproof
\end{theorem}
Note that in the analogue of Corollary~\ref{crreg} this means that
we consider a random subgraph of a random $r$-regular simple graph.
Here one must be slightly careful with the argument: we need to explore
the subgraph, but then check whether the \emph{original} graph is simple.

The rest of the paper is organized as follows.
In the next section we define the exploration process that we
study, and two corresponding random walks $(X_t)$ and $(Y_t)$.
In Section~\ref{sec_traj} we establish some key properties of $(X_t)$,
including the `idealized trajectory' that we expect it to remain close to.
Using these properties, and assuming Theorem~\ref{th-1}
for the moment, we prove Theorem~\ref{th0} in Section~\ref{sec_crit}
and Theorem~\ref{th1} in Section~\ref{sec_above}.
In Section~\ref{sec_local} we prove a local limit theorem (Lemma~\ref{llt}) for
certain sums of independent random variables.
Finally, in Section~\ref{sec_below} we prove Theorem~\ref{th-1}.
The proof (which uses the local limit theorem)
is rather different from that of the other main
results: we use domination arguments to study the initial behaviour of
$(X_t)$, rather than following its evolution as in the main part of the
paper. This is the reason for postponing the proof, even though (a weak form of)
the result is needed in the proofs of Theorems~\ref{th1} and~\ref{th0},
to rule out `other' large components.

\section{The exploration process}\label{sec_explore}

Consider the following exploration process for uncovering the
components of $G=G_{\vd}^\mm$. This is a slight variant of the
standard process, in that we check for `back-edges' forming cycles as
we go. A form of this variant was used by Ding, Kim, Lubetzky and
Peres~\cite{DKLP_diam} in their study of the diameter of $G(n,p)$;
we could in fact use a more standard exploration here, but the
variant results in slightly cleaner calculations.

We shall define an exploration so that after $t$ steps of the process,
$t$ vertices have been `reached'; the other $n-t$ are `unreached'.
Furthermore, a certain random number of stubs will have been paired with each
other, and each unpaired stub will be either `active' or `unreached';
the active stubs are attached to reached vertices, the unreached ones
to unreached vertices.
We write $A_t$ for the number of active stubs, and $U_t$ for the number
of unreached stubs. The process we define will be such that
in the complete pairing $\pP$, the active stubs are paired to a subset of the unreached stubs,
and the remaining unreached stubs are paired with each other. Moreover, the conditional distribution
of the pairing $\pP$ given the first $t$ steps of the process is such that
all pairings of the active and unreached stubs satisfying this condition are equally likely.

At step $t+1$ of the process, if $A_t>0$ then we pick an active stub $a_{t+1}$, for example the first
in some order fixed in advance. Then we reveal its partner $u_{t+1}$ in the random pairing $\pP$,
which is necessarily unreached. Let $v_{t+1}$ be the corresponding unreached vertex.
The stubs $a_{t+1}$ and $u_{t+1}$ are now paired (so in particular $a_{t+1}$ is no longer active),
and
the remaining stubs $u_{t+1,1},\ldots,u_{t+1,r_{t+1}}$ attached to $v_{t+1}$ are provisionally declared active.
But now 

(i) we check whether any other (previously) active stubs are paired to any of the $u_{t+1,i}$,
and then 

(ii) we check whether any of the remaining $u_{t+1,i}$ are paired to each other.

\noindent
We declare any pairs found in (i) or (ii) `paired', and continue. These pairs
correspond to `back-edges'. 

If $A_t=0$ then we simply pick $v_{t+1}$ to be a random unreached vertex, chosen
with probability proportional to degree, provisionally declare
all its stubs active, and then perform the second check (ii) above.
In this case we say that we `start a new component' at step $t+1$.

For (partial) compatibility with the notation in \cite{BR_walk}, let $\eta_{t+1}$ denote
the degree of $v_{t+1}$. We write $\newbe_{t+1}$ for the number of back-edges
found during step $t+1$, and $Y_t=\sum_{i\le t}\newbe_i$ for the total number
of back-edges found during the first $t$ steps.

Let $C_t$ be the number of components that we have started exploring within the first
$t$ steps, and set
\begin{equation}\label{Xdef}
 X_t = A_t-2C_t.
\end{equation}
Considering separately the cases $A_t>0$ and $A_t=0$, and noting that finding
a back-edge pairs off two stubs, we see that
\begin{equation}\label{Xdiff}
 X_{t+1} - X_t = \eta_{t+1}-2 -2\newbe_{t+1},
\end{equation}
while by the definition of $Y_t$,
\begin{equation}\label{Bdiff}
 Y_{t+1} - Y_t = \newbe_{t+1}.
\end{equation}

Let $\cF_t$ denote the (finite, of course) sigma-field generated by the information
revealed by step $t$.
Let $U_{d,t}$ denote the number of unreached \emph{vertices} of degree $d$ after $t$ steps,
so, recalling that $U_t$ denotes the (total) number of unreached \emph{stubs}, we have
\begin{equation}\label{Udef}
 U_t = \sum_d dU_{d,t}.
\end{equation}
In both cases above, the vertex $v_{t+1}$ is chosen from the unreached
vertices and, given $\cF_t$, the probability that any given vertex is chosen is proportional
to its degree.
Hence
\begin{equation}\label{etaprob}
 \Pr(\eta_{t+1}=d \mid \cF_t) = dU_{d,t} /U_t.
\end{equation}
In particular,
\[
 \E(\eta_{t+1}-1 \mid \cF_t) = U_t^{-1} \sum_d d(d-1)U_{d,t}.
\]

In the analysis that follows, we shall impose the assumption
\begin{equation}\label{assump}
 t\le c_0 n/2,
\end{equation}
where $c_0$ is the constant in \eqref{a2}.
Since at least $c_0 n$ vertices have degree at least~$1$,
\eqref{assump} implies that
\begin{equation}\label{ubelow}
 U_t\ge c_0 n/2,
\end{equation}
and in particular that $U_t=\Theta(n)$. This will simplify some formulae
in the calculations.

Suppose that $A_t>0$. Then, given $\cF_t$ and $v_{t+1}$, the expected
number of pairs discovered during check (i) above is exactly
\begin{equation*}
 (A_t-1)\frac{\eta_{t+1}-1}{U_t-1} = (\eta_{t+1}-1)A_t/U_t +O(1/n),
\end{equation*}
using $U_t=\Theta(n)$ for the approximation. Indeed, each of the $A_t-1$ other
stubs that were active before this step is equally likely to be paired
to any of the $U_t-1$ remaining unreached stubs.
We could write an exact formula for the expected number of pairs
found during check (ii), but there is no need: instead we simply note that
the expectation is $O(1/n)$.
It follows that
\begin{equation}\label{pform}
 \E(\newbe_{t+1}\mid \cF_t,\eta_{t+1}) =  (\eta_{t+1}-1)A_t/U_t +O(1/n),
\end{equation}
and hence that
\begin{equation}\label{Ebeta}
 \E(\newbe_{t+1}\mid \cF_t) = \E(\eta_{t+1}-1\mid \cF_t)A_t/U_t +O(1/n).
\end{equation}
The last two formulae are also valid when $A_t=0$: this time there is no check (i), and
the expected number of back-edges found during check (ii) is $O(1/n)$.
From \eqref{Xdiff} and \eqref{Xdef} it follows that
\begin{equation}\label{Xdelta}
 \E( X_{t+1} - X_t \mid \cF_t ) = -1 + \E(\eta_{t+1}-1\mid \cF_t)(1-2X_t/U_t) + O(C_t/n).
\end{equation}
We use $C_t\ge 1$ for $t\ge 1$ to avoid writing a separate $O(1/n)$ error term, not that
it matters.

Simply knowing the expected changes at each step is good enough to allow us to deduce
fairly tight bounds on the size of the giant component. But for asymptotic
normality we need a bound on the variance. 
We could give a formula that is useful when $A_t$ is comparable with $U_t$,
but we shall not need this. Instead, we simply note that the number $\newbe_{t+1}$
of pairs found during our checks (i) and (ii) is bounded, and for $t=o(n)$,
which implies $A_t=o(U_t)$, we have $\newbe_{t+1}=0$ whp.
Since $\eta_{t+1}$ is bounded, it follows easily that for $t=o(n)$ we have
\[
 \Var(X_{t+1}-X_t\mid \cF_t) = \Var(\eta_{t+1}\mid \cF_t) +o(1) = v_0+o(1),
\]
where $v_0$ is defined in \eqref{v0def}, and the second equality follows
from the fact that only $o(n)$ vertices have been `used up' by time $t$.

Since $A_t\ge 0$, with equality only when we have just finished
exploring a component, the times $t_1,\ldots,t_{c(G)}$ at which we finish exploring
components
are given by
\begin{equation}\label{ti}
 t_i = \min\{ t: X_t=-2i\}.
\end{equation}
Since exactly one vertex is revealed at each stage, if $\cC_i$
denotes the $i$th component explored, then
\[
 |\cC_i| = t_i-t_{i-1},
\]
where we set $t_0=0$.

Recall that $Y_t$ denotes the number of back-edges found within the first $t$ steps.
Then $Y_{t_i}-Y_{t_{i-1}}$ is simply the nullity of $\cC_i$:
\begin{equation}\label{nullity}
 n(\cC_i) = Y_{t_i}-Y_{t_{i-1}}.
\end{equation}

In the rest of the paper we shall study the behaviour of the random
walks $(X_t)$ and $(Y_t)$, and use this to prove our main results.

\section{Trajectory and deviations}\label{sec_traj}

As in~\cite{BR_walk}, we shall write the difference $X_{t+1}-X_t$ as $D_t+\Delta_t$,
where $D_t=\E(X_{t+1}-X_t\mid\cF_t)$, so the $\Delta_t$ may
be regarded as a sequence of martingale differences.
It is more or less automatic that $\sum_{i\le t}\Delta_i$
is asymptotically normally distributed, so we need to understand
the sum of the $D_t$. Each $D_t$ depends on $X_t$,
but it turns out that the dependence is not very strong.
So if we can find an `ideal' trajectory (corresponding to
all $\Delta_t$ being equal to zero), then 
bounding the deviations of $(X_t)$ from this trajectory will not be too difficult.

The first problem is that the term $U_t$ appearing in \eqref{etaprob} is not so simple.
We start with some lemmas.
The first is a standard result about order statistics.

\begin{lemma}\label{O1}
Let $Z_1,\ldots,Z_n$ be i.i.d.\ samples from a distribution $Z$ on $[0,1]$
with distribution function $G(x)=\Pr(Z\le x)$, and let $N_n(x)=|\{i:Z_i\le x\}|$.
Then for any (deterministic) function $y=y(n)$ we have
\begin{equation}\label{even}
 \sup_{0\le x\le y} |N_n(x)-nG(x)| = \Op(\sqrt{nG(y)}).
\end{equation}
\end{lemma}
\begin{proof}
Let $A$ have a Poisson distribution with mean $n$, and given $A$, let $Z_1',\ldots,Z_A'$
be i.i.d.\ with distribution $Z$, so the set $\{Z_i'\}$ forms a Poisson process
on $[0,1]$. (If $Z$ has a density function $g(x)=G'(x)$,
then the intensity measure of the Poisson process is $ng(x)\dd x$.)
Writing $N'(x)$ for the number of $Z_i'$ in $[0,x]$, 
consider the random function $F(x)=N'(x)-nG(x)$. From basic properties
of Poisson processes, this function is a continuous-time martingale on $[0,1]$ with independent
increments.
Hence Doob's maximal inequality~\cite[Ch. III, Theorem 2.1]{Doob}
gives
\[
 \Pr(\sup_{x\le y} |F(x)|\ge t)  \le \frac{\Var(F(y))}{t^2} = \frac{nG(y)}{t^2},
\]
since $\Var(F(y))=\Var(N'(y))$, and up to an additive constant, $N'(y)$
is simply Poisson with mean $nG(y)$.
It follows that
\begin{equation}\label{evenP}
 \sup_{x\le y}|F(x)| = \Op(\sqrt{nG(y)}).
\end{equation}

The expected value of $|A-n|$ is $O(\sqrt n)$. When $A>n$, delete a random
subset of the points $\{Z_i'\}$ of size $A-n$. When $A<n$, add
$n-A$ i.i.d.\ new points to $\{Z_i'\}$ with distribution $Z$.
In this way we obtain a set of $n$ i.i.d.\ samples from $Z$.
Given $A$, the added/deleted points have distribution $Z$,
so the expected number in $[0,y]$ is $|A-n|G(y)$. Hence
the unconditional expectation of the number of points
added or deleted in $[0,y]$ is $O(\sqrt{n}G(y))=O(\sqrt{nG(y)})$.
Hence \eqref{even} follows from \eqref{evenP}.
\end{proof}

In our exploration process, the next vertex $v_{t+1}$ is always chosen with
probability proportional to its degree. Hence the (distribution of)
the random sequence $\sigma=(v_1,\ldots,v_n)$
has the following alternative description:
first assign a random order to all stubs. Then sort the vertices so that $v$
comes before $w$ if and only if $v$'s earliest stub comes before
$w$'s earliest stub.
In turn, we may realize the random order on the stubs by assigning i.i.d.\ $U[0,1]$
variables to the stubs; we shall call these variables \emph{stub values}.

For each $t$, there is a random `cut-off' $Z_t\in [0,1]$ so that a vertex $v$
is among $v_1,\ldots,v_t$ if and only if its smallest stub value
is at most $1-Z_t$. Fixing a cut-off value $z$, the expected
number of vertices with all stub values at least $1-z$ is exactly
$\sum_d n_d z^d$,
so we expect to have
\[
 n-t \sim \sum_d n_dZ_t^d.
\]

Let $f(z)=f_n(z)$ denote the probability generating function of the degree
distribution $\vd$,
so
\begin{equation}\label{fdef}
 f(z) = \sum_{d=1}^{\dmax} \frac{n_d}{n} z^d.
\end{equation}
Recall that $U_{d,i}$ denotes the number of unreached degree-$d$ vertices
after $i$ steps, i.e., the number of degree-$d$ vertices
\emph{not} among the first $i$ elements of our random order on the vertices.
\begin{theorem}\label{left}
Let $\vd=\vd_n$ be any degree sequence of length $n$
with all degrees between $1$ and some constant $\dmax$, and let $(v_1,\ldots,v_n)$
be the random order on the vertices defined above.
Define a function $\tau\mapsto z(\tau)$ from $[0,1]$ to $[0,1]$ by $f(z(\tau))+\tau=1$,
where $f$ is the probability generating function of $\vd$.
Then for any $t=t(n)\ge 1$ we have
\[
 \max_{1\le d\le \dmax}\max_{1\le i\le t} |U_{d,i} - n_d z(i/n)^d| = \Op(\sqrt{t}),
\]
where $n_d$ is the number of degree-$d$ vertices in $\vd$.
\end{theorem}
\begin{proof}
Construct the sequence $v_1,\ldots,v_n$ from stub values as above.
Let $Z^{(d)}$ denote the distribution obtained by taking the minimum of
$d$ independent $U[0,1]$ random variables. 
Note that $Z^{(d)}$ has distribution function $G_d(x)=1-(1-x)^d$.

For each $d$, let $Z_{d,i}$ denote the minimum stub value
of the $i$th degree-$d$ vertex. Then $Z_{d,1},\ldots,Z_{d,n_d}$ are i.i.d.\ with
distribution $Z^{(d)}$.
Let $m_d(x)$ denote the number of degree-$d$ vertices whose
smallest stub value is at most $x$. Then by Lemma~\ref{O1}, for any (deterministic) $y=y(n)$
and any $d$ we have
\[
 E_{d,y} = \sup_{0\le x\le y} |m_d(x) - n_dG_d(x)| = \Op(\sqrt{n_dG_d(y)}).
\]
Summing over $d$, and using $\sum_{i=1}^k \sqrt{a_i}\le \sqrt{k} \sqrt{\sum a_i}$, we have
\[
 E_y = \sum_d E_{d,y} = \Op(\sqrt{nG(y)}),
\]
where
\begin{equation}\label{gf}
 G(x)=\sum_{d=1}^{\dmax}\frac{n_d}{n}G_d(x) = 1-f(1-x).
\end{equation}
Let $m(x)=\sum_d m_d(x)$. Then by the triangle inequality
we have
\begin{equation}\label{msup}
 \sup_{0\le x\le y} |m(x) - nG(x)| \le E_y.
\end{equation}

Define $y$ by $nG(y)=t$, so $E_y=\Op(\sqrt{t})$.
For $1\le i\le t$ set $x_i=1-z(i/n)$. From \eqref{gf} we have $G(x_i)=1-f(z(i/n))$,
so by the definition of the function $z$, we have $G(x_i)=i/n$,
i.e., $nG(x_i)=i$.
Note that $x_1<x_2<\cdots<x_t=y$.

From \eqref{msup}, for every $i\le t$ we have $|m(x_i)-i|\le E_y$.
Furthermore, the number of degree-$d$ vertices among the first $m(x_i)$ vertices,
namely $m_d(x_i)$, differs from $n_dG_d(x_i)$ by at most $E_{d,y}\le E_y$.
It follows that the number of degree-$d$ vertices among the first $i$ vertices
differs from $n_dG_d(x_i)$ by at most $2E_y$,
i.e.,
\[
 | (n_d-U_{d,i}) - n_dG_d(x_i) | \le 2E_y.
\]
The difference on the left is exactly $|U_{d,i}-n_dz(i/n)^d|$, so the result follows.
\end{proof}

Returning to our process, recall that the vertices $v_1,v_2,\ldots$ are chosen
according to the random distribution considered above.
Recall also that when \eqref{assump} holds, then $U_t=\Theta(n)$.

\begin{corollary}\label{c2}
Define $z(\tau)$ by $f(z(\tau))+\tau=1$.
Then for any $t\le c_0n/2$, writing $z$ for $z(i/n)$, we have
\begin{equation}\label{Ubd}
 \sup_{i\le t} \left| \frac{U_i}{n} - zf'(z) \right| =\Op(\sqrt{t}/n)
\end{equation}
and
\begin{equation}\label{etabd}
 \sup_{i\le t} \left|\E\bb{\eta_{i+1}-1\mid \cF_i} - \frac{zf''(z)}{f'(z)} \right| = \Op(\sqrt{t}/n).
\end{equation}
Furthermore,
if $t=o(n)$, then
\[
 \sup_{i\le t} \left|\Var\bb{\eta_{i+1}-1\mid \cF_i} - v_0 \right| = O(t/n) =o(1),
\]
where $v_0$ is defined by \eqref{v0def}.
\end{corollary}
\begin{proof}
By Theorem~\ref{left} there is some random $K_t$ such that $K_t=\Op(\sqrt{t})$
and, for all $i\le t$ and all $d\le\dmax$,
\begin{equation}\label{ibd}
 |U_{d,i} - n_d z(i/n)^d| \le K_t.
\end{equation}
Fix $i$ and write $z$ for $z(i/n)$.
Noting that $zf'(z)=n^{-1}\sum_d d n_d z^d$, \eqref{Udef} and \eqref{ibd} give
\[
 |U_i/n-zf'(z)| = \frac{1}{n}\left| \sum_{d=1}^{\dmax} dU_{d,i} - \sum_d dn_d z^d\right| \le \dmax^2 K_t/n.
\]
Since $K_t=\Op(\sqrt{t})$, this proves \eqref{Ubd}.

For \eqref{etabd}, note that $\Pr(\eta_{i+1}=d\mid \cF_i)$ is by definition
equal to $dU_{d,i}/U_i$ (see \eqref{etaprob}), so when \eqref{ibd} holds,
$\Pr(\eta_{i+1}=d\mid \cF_i)$ is within $O(K_t/n)$
of
\[
 q_d(z) = \frac{d n_d z^d}{\sum_{d'} d' n_{d'}z^{d'}}.
\]
The bound \eqref{etabd} follows by noting that
\[
 \sum_d (d-1)q_d(z) = \frac{\sum d(d-1)n_d z^d}{\sum_d d n_dz^d} = \frac{z^2f''(z)}{zf'(z)} = \frac{zf''(z)}{f'(z)}.
\]
The variance bound is immediate from the fact that by time $t$ we have `used up' $O(t)$ vertices
of each degree, changing the conditional variance by at most $O(t/n)$.
\end{proof}

Let us now define the idealized trajectory that we have in mind.
Recall that $f(z)=f_n(z)$ is defined by \eqref{fdef}, and
$z=z(\tau)$ by
\begin{equation}\label{ztau}
 f(z)+\tau =1
\end{equation}
for $0\le \tau\le 1$, so
\begin{equation}\label{dztau}
 \frac{\dd z}{\dd \tau}= -\frac{1}{f'(z)}.
\end{equation}
We shall think of $\tau$ as a rescaled time parameter, taking $\tau=t/n$,
but will write our trajectory as a function of $z$.
In the light of our assumption~\eqref{assump}, we need only consider $\tau\le c_0/2$.
It is easy to check that this implies $z\ge 1/2$. In fact, we can always assume
that $\tau=o(1)$ and so $z=1-o(1)$.

Let $\mu_1=f'(1)=n^{-1}\sum dn_d$ be the average degree in our graph $G=G_{\vd}^\mm$, and
define functions $x(\tau)$ and $u(\tau)$ by
\begin{equation}\label{xdef}
 x = zf'(z)-\frac{f'(z)^2}{\mu_1}
\end{equation}
and
\begin{equation}\label{udef}
 u = zf'(z).
\end{equation}
Using elementary calculus, it is straightforward to check that these functions
satisfy $x=0$ when $\tau=0$, i.e., when $z=1$, and
\begin{equation}\label{DE}
 \frac{\dd x}{\dd \tau} = \frac{\dd x}{\dd z}\frac{\dd z}{\dd \tau} = -1+\frac{z f''(z)}{f'(z)}\left(1-\frac{2x}{u}\right).
\end{equation}

Recall that $f$ may depend on $n$. Since at least $c_0n$ vertices have degree
between $1$ and $\dmax$, for $z\ge 1/2$ (which is the only range we consider), we have $f'(z)\ge c_02^{-\dmax}$,
so $f'(z)$ is bounded below away from zero. Also, any given derivative
of $f(z)$ is bounded by a constant depending only on $\dmax$.
Using \eqref{dztau} it follows easily that
the derivative of any fixed order of $x$ with respect to $\tau$
is bounded uniformly in (large enough) $n$.

One can check that
\[
 -(f')^3\mu_1\frac{\dd^2 }{\dd\tau^2} x = 2(f')^2f'''+\mu_1 z (f'')^2-\mu_1 f'f''-\mu_1 z f'f'''.
\]
Substituting $z=1$ and noting that $\mu_i=f^{(i)}(z)|_{z=1}$, we have
\[
 -\mu_1^4\left.\frac{\dd^2 }{\dd\tau^2} x\right|_{z=1}
  = 2\mu_1^2\mu_3+\mu_1 \mu_2^2-\mu_1^2 \mu_2-\mu_1^2 \mu_3 = \mu_1(\mu_1\mu_3+\mu_2^2-\mu_1\mu_2).
\]
By assumption $\la=\mu_2/\mu_1\sim 1$, and as noted in Section~\ref{sec_intro},
$\mu_3=\Theta(1)$ (see~\eqref{vo1}). Hence
\begin{equation}\label{xx0}
 \left.\frac{\dd^2 }{\dd\tau^2} x\right|_{\tau=0} \sim -\frac{\mu_3}{\mu_1^2} = -\Theta(1).
\end{equation}

Let us collect together some basic properties of the trajectory $x$. We write
$\xdot$ for the derivative of $x$ with respect to $\tau$.

\begin{lemma}\label{traj}
Suppose that $\vd$ satisfies the assumptions \eqref{a1} and \eqref{a2}, and that
$\la=\la(\vd)\to 1$. Then

\begin{romenumerate}
\item $x(0)=0$,

\item $\xdot(0)=\la-1$ and

\item $\ddot{x}(\tau) = -\frac{\mu_3}{\mu_1^2} +o(1)$, uniformly in $\tau=o(1)$.
\newcounter{local}
\setcounter{local}{\value{enumi}}
\end{romenumerate}
\noindent
Suppose in addition that $\eps=\la-1>0$ and that $\eps^3 n\to\infty$,
and let $\rho=\rho_n$ be defined by \eqref{rhodef}. Then also
\begin{romenumerate}
\setcounter{enumi}{\value{local}}
\item $x(\rho)=0$,

\item\label{v} $\rho\sim\frac{2\mu_1^2}{\mu_3}\eps$,

\item $\xdot(\rho)\sim -\eps$ and

\item $x(\tau)>\eps \tau/2$ whenever $0\le\tau=o(\eps)$.
\end{romenumerate}
\end{lemma}

\begin{proof}
We have noted (i) already. Substituting $z=1$ (corresponding to $\tau=0$) into \eqref{DE} gives 
(ii). (iii) follows from \eqref{xx0} and the fact (noted above) that the third
derivative of $x$ is uniformly bounded over $\tau=o(1)$.

Turning specifically to the supercritical case,
(iv) follows easily from \eqref{xdef} and \eqref{rhodef}.
Indeed, recalling that $f$ is the generating function of our
degree distribution, \eqref{prerho} says exactly that
$z$ is the smallest positive solution to $z=f'(z)/\mu$.
From \eqref{xdef} we have $x=0$ at this value of $z$.
Now \eqref{rhodef} says that $\rho=1-f(z)$ which, by our change
of variable formula \eqref{ztau}, is exactly the corresponding
value of $\tau$.

(v)--(vii) follow from (i)--(iv) and Taylor's Theorem.
\end{proof}

For $1\le t\le c_0n/2$ set $x_t=n x(t/n)$ and $u_t=n u(t/n)$;
our aim is to show that $X_t$ will be close to $x_t$ and
$U_t$ close to $u_t$; the functions $x$ and $u$ are the corresponding
`scaling limits'.
Since, as a function of $\tau$, $x$ has uniformly bounded second derivative,
we have
\[ 
 x_{t+1}-x_t = \left.\frac{\dd x}{\dd \tau}\right|_{\tau=t/n} +O(1/n).
\]
Hence, from \eqref{DE}, writing $z$ for $z(t/n)$, and defining $w_t=\frac{zf''(z)}{f'(z)}$, we have
\begin{equation}\label{xdelta}
 x_{t+1}-x_t = -1+w_t\left(1-\frac{2x_t}{u_t}\right) +O(1/n),
\end{equation}
which is strongly reminiscent of \eqref{Xdelta}.
Our aim is to show that $(X_t)$ will whp remain close to $(x_t)$,
and use this, and the asymptotic normality of the deviations,
to prove Theorem~\ref{th1}.

For the rest of the section we fix some $\tmax=\tmax(n)=o(n)$ (see the next
two sections for specific values).
In what follows, we shall only consider values of $t$ up to $\tmax$.

Set
\[
 D_t =  \E( X_{t+1} - X_t \mid \cF_t ),
\]
noting that $D_t$ is random.
Corollary~\ref{c2} shows that for any deterministic $t=t(n)$ with $t\le \tmax$,
there is some random $K_t$ satisfying
\begin{equation}\label{Abd}
 K_t=\Op(\sqrt{t}/n)
\end{equation}
such that $|U_i-u_i|/n\le K_t$ and $|\E(\eta_{i+1}-1 \mid \cF_i) - w_i|\le K_t$
for all $i\le t$. Recalling that $U_i\ge c_0 n/2$
in the range we consider, and noting that $\eta_{i+1}$ is bounded by $\dmax$,
it follows using \eqref{Xdelta} and \eqref{xdelta} that for $i\le t\le \tmax$ we have
\begin{equation}\label{Dclose}
 |D_i-(x_{i+1}-x_i)| \le c\left(\frac{|X_i-x_i|}{n} +K_t + \frac{C_i}{n}\right),
\end{equation}
for some constant $c$ that depends only on $\dmax$ and $c_0$.
Since $X_0=0$, we may write $X_t$ as
\[
 X_t = \sum_{i<t} (D_i+\Delta_i) = \sum_{i<t} D_i + S_t,
\]
where
\[
 S_t = \sum_{i<t}\Delta_i.
\]

Let
\begin{equation}\label{tXdef}
 \tX_t = x_t + S_t,
\end{equation}
which we shall think of as a (rather precise) random approximation to $X_t$,
and define the `error term' $E_t$ by
\begin{equation}\label{Etdef}
 E_t=X_t-\tX_t = \sum_{i<t} D_i -x_i.
\end{equation}
 
Recall that $x_t$ is deterministic. The key point is that the
distribution of $S_t$ is easy to control, since
$(S_t)$ is a martingale with bounded differences.
Let
\[
 M_t = \max_{0\le i\le t} |S_i|.
\]
\begin{lemma}\label{Smax}
For any (deterministic) $t=t(n)$ and any $m\ge 0$
we have
\begin{equation}\label{Mb1}
 \Pr(M_t\ge m)\le \dmax^2 t/m^2.
\end{equation}
In particular, $M_t=\Op(\sqrt{t})$.
Furthermore, for any $\alpha=\alpha(n)>0$ and any $\omega=\omega(n)\to\infty$,
the event $\{ M_t\le \alpha t/4 $ for all $\omega/\alpha^2\le t\le \tmax\}$ holds whp.
\end{lemma}
\begin{proof}
Since the differences $\Delta_i$ are bounded by $\dmax$, their (conditional) variances
are at most $\dmax^2$, so $\Var(S_t)\le \dmax^2 t$.
Applying Doob's maximal inequality gives \eqref{Mb1}. That $M_t=\Op(\sqrt{t})$
follows immediately.

For the last part, let $t_i=2^i\omega/\alpha^2$, and let $\cE_i$ be the event
that $M_{t_{i+1}}\ge \alpha t_i/4$. Then \eqref{Mb1} gives
\[
 \Pr(\cE_i)\le \frac{\dmax^2t_{i+1}}{\alpha^2t_i^2/16} \le \frac{32\dmax^2}{2^i\omega}.
\]
It follows that $\sum \Pr(\cE_i)=o(1)$, so whp no $\cE_i$ holds, giving the result.
\end{proof}
\begin{lemma}\label{norm}
Let $t=t(n)$ satisfy $t\to\infty$ and $t=o(n)$.
Then $S_t$ is asymptotically normal with mean $0$ and variance $v_0t$,
where $v_0$ is given by \eqref{v0def}.
\end{lemma}
\begin{proof}
Recall that $(S_t)$ is a martingale with $S_0=0$,
and the differences $\Delta_i$ are bounded.
Moreover, by Corollary~\ref{c2},
$\Var(\Delta_{i+1}\mid \cF_i)\sim v_0$ when $i=o(n)$.
The result thus follows from a standard martingale central limit theorem 
such as Brown~\cite[Theorem 2]{Brown}.
\end{proof}

We now turn to the error terms $E_t$.
Using the second expression for $E_t$ in \eqref{Etdef},
summing \eqref{Dclose} over $1\le i<t$, and noting
that $|X_i-x_i|=|S_i+E_i|$, for $t\le \tmax$ we see that
\begin{eqnarray*}
 |E_t| &\le& c\left(\frac{t}{n}\max_{i<t}|X_i-x_i| +tK_t+\frac{tC_{t-1}}{n}\right) \\
  &\le& \frac{c t}{n}\max_{i<t}|E_i| +c\left(\frac{tM_t}{n} +tK_t+\frac{tC_{t-1}}{n}\right).
\end{eqnarray*}
Since $ct/n\le c\tmax/n=o(1)$ is less than $1/2$
if $n$ is large, then for $n$ large enough (which we assume
from now on), it follows by induction on $i$
that
\begin{equation}\label{Ei1}
 |E_i| \le 2c\left(\frac{tM_t}{n}+tK_t+\frac{tC_{t-1}}{n}\right)
\end{equation}
for $0\le i\le t$.

Set
\begin{equation}\label{M*def}
 M^*_t = 2c\left(\frac{tM_t}{n} + tK_t\right).
\end{equation}
Note that $M^*_t$ is increasing in $t$. Also, for $0\le t\le \tmax$,
\eqref{Ei1} (applied with $i=t$) gives
\begin{equation}\label{Ei1a}
 |E_t| \le M^*_t + 2ct C_{t-1}/n.
\end{equation}
For any (deterministic) $t=t(n)$,
it follows from \eqref{Abd} and Lemma~\ref{Smax} that
\begin{equation}\label{M*}
 M^*_t = \Op(t^{3/2}/n).
\end{equation}

\section{The critical case}\label{sec_crit}

Using the bounds from the previous section, it is very easy to prove Theorem~\ref{th0}.
The hardest part is establishing that the description of the limit actually makes
sense; this follows from the results of Aldous~\cite{Aldous} by simple rescaling.
Since all probabilistic technicalities are same as in~\cite{Aldous}, we shall
not mention them. Indeed, we take a combinatorial point of view in the estimates
that follow.

\begin{proof}[Proof of Theorem~\ref{th0}]
Recall that by assumption $n^{1/3}(\la-1)\to\alpha_1\in \R$, while $\mu_3/\mu_1\to\alpha_0$
and $\mu_3/\mu_1^2\to\alpha_2$ with $\alpha_0,\alpha_2>0$.
For the moment, let $\omega$ be a large constant. A little later
we shall allow $\omega$ to tend to infinity slowly.

Define a random function $S(s)$ on $[0,\omega]$ by setting $S(t/n^{2/3})=n^{-1/3} S_t$
for $0\le t\le \omega n^{2/3}$, and interpolating linearly between these values.
Recall that $S_t$ is a martingale with bounded differences
and that for $t\le \omega n^{2/3}=o(n)$, the conditional variances
of the differences are $v_0+o(1)\sim \mu_3/\mu_1 \sim \alpha_0$ (see \eqref{vo1}).
It follows easily that $S$ converges to $\alpha_0^{1/2} W$ where
$W$ is a standard Brownian motion on $[0,\omega]$, in the sense that these
random functions can be coupled so that $\sup_{s\in [0,\omega]} |S(s)-\alpha_0^{1/2}W(s)|$
converges to 0 in probability.
(To see this, one can apply a functional martingale central limit theorem,
or simply subdivide into suitable short intervals and use a standard
martingale CLT.)

Recalling that $\tX_t=x_t+S_t=nx(t/n)+S_t$,
define a random function $\tX(s)$ on $[0,\omega]$ by setting
\[
 \tX(s)=n^{-1/3} \tX_t = n^{2/3}x(sn^{-1/3}) + S(s)
\]
whenever $s=t/n^{2/3}$ for integer $t\le \omega n^{2/3}$, and again interpolating
linearly. Note that $\tX$ is given by adding a deterministic function to $S$.
From Lemma~\ref{traj} and Taylor's theorem, we have
\begin{eqnarray*}
 n^{2/3}x(sn^{-1/3}) &=& 0 + n^{2/3}(\la-1)sn^{-1/3} - n^{2/3}\left(\frac{\mu_3}{\mu_1^2}+o(1)\right) \frac{s^2 n^{-2/3}}{2} \\
 &=& \alpha_1 s - \frac{\alpha_2}{2} s^2 +o(1),
\end{eqnarray*}
uniformly in $s\in [0,\omega]$. It follows that
$\tX$ converges to the inhomogeneous Brownian motion $W_\alphas$ defined in \eqref{Wdef}.

Setting $\tmax=\omega n^{2/3}$, by \eqref{M*} the quantity $M^*_{\tmax}$ defined in \eqref{M*def}
satisfies
\[
 M^*_{\tmax}=\Op(\tmax^{3/2}/n)=\Op(1)=\op(n^{1/3}).
\]
Pick some (deterministic)
$k=k(n)$ with $k/n^{1/3}\to\infty$.
Suppose that we explore more than $k$ components in the first $\tmax$ steps.
When we finish exploring the $k$th component we have $X_t=-2k$
and $C_t=k$. Since $\tmax=o(n)$, the bound \eqref{Ei1a} thus gives
$|E_t|\le \Op(1)+o(k)$. In particular, whp $|E_t|\le k$.
It then follows that $|\tX_t|\ge |X_t|-|E_t|\ge k$.
From the convergence of $(\tX_t)$ to $W_{\alpha_0,\alpha_1,\alpha_2}$
we have $\sup_{t\le \tmax}|\tX_t| = \Op(n^{1/3})$.
It follows that $C_{\tmax}=\Op(n^{1/3})$.
Using \eqref{Ei1a} again, this gives $\sup_{t\le \tmax}|E_t|=\Op(1)$.
In other words, the `idealized random trajectory' $\tX_t$
is an extremely close approximation to $X_t$ up to time $\tmax$.
Using $(X_t)$ to defining a function $X$ on $[0,\omega]$ as for $S$ and $\tX$ above,
it follows that $X$ also converges to $W_\alphas$.

Finally, recalling that $X_t=A_t-2C_t$, and that $X_t$ first hits
$-2k$ when we finish exploring the $k$th component,
i.e., just before $C_t$ increases to $k+1$, it is easy to check that for $t>0$
we have $A_t = X_t-\min_{i<t}X_i +O(1)$.
Defining a final function $A$ on $[0,\omega]$ using the $A_t$,
we see that $A$ converges to the function $B$ defined in \eqref{Bdef}.
So far $\omega$ was fixed, but convergence for all fixed $\omega$
implies convergence for $\omega\to\infty$ sufficiently slowly.
Now the sizes of the components explored during the first $\omega n^{2/3}$
steps are simply $n^{2/3}$ times the excursion lengths of $A$,
which converge to the excursion lengths of $B$. (This follows from
basic properties of $B$.)

For the component sizes, it remains only to show that 
when $\omega\to\infty$, for any $\delta>0$,
whp there are no components of size at least $\delta n^{2/3}$ in the rest
of the graph. This follows from Theorem~\ref{th-1} by an argument
similar to that at the end of Section~\ref{sec_above}.

Finally, we also claimed convergence for the nullities, or numbers of back-edges,
to appropriate Poisson parameters. For $t=o(n)$,
which implies $U_t\sim \mu_1 n$ and $\E(\eta_{t+1}-1\mid \cF_t)\sim \E(\eta-1)=\la\sim 1$,
from \eqref{Ebeta} we have
\[
 \E(\newbe_{t+1}\mid \cF_t) = (1+o(1))\frac{A_t}{\mu_1 n} +O(1/n).
\]
In terms of the rescaled function $A$ on $[0,\omega]$, this corresponds
to formation of back-edges at rate $A(s)/\mu_1$, and 
joint convergence of the component sizes and back-edge counts
to the excursion lengths and mark counts claimed in the theorem follows
easily as in~\cite{Aldous}.
\end{proof}

\section{The supercritical case}\label{sec_above}

We follow the argument in~\cite{BR_walk} closely, and attempt to use the same notation
where possible.
Throughout this section we fix a function $\omega=\omega(n)$ tending to infinity slowly,
in particular with $\omega^6=o(\eps^3 n)$ and $\omega^2=o(1/\eps)$. As in~\cite{BR_walk},
set
\[ 
 \sigma_0=\sqrt{\eps n}
\]
and
\[
 t_0=\omega\sigma_0/\eps,
\]
ignoring, as usual, the irrelevant rounding to integers. Note for later that $t_0=o(\eps n)$,
and that $t_0\ge \omega/\eps^2$ if $n$ is large.
We shall apply the results of Section~\ref{sec_traj} with, say
\begin{equation}\label{tmax}
 \tmax=4 \frac{\mu_1^2}{\mu_3}\eps n \sim 2\rho n =\Theta(\eps n);
\end{equation}
see Lemma~\ref{traj}\ref{v}.

\begin{lemma}\label{start}
Let $Z$ denote the number of components completely explored
by time $t_0$, and let $T_0=\inf\{t:X_t=-2Z\}$ be the time at which
we finish exploring the last such component.
Then $Z\le \omega/\eps\le \sigma_0/\omega$ and $T_0\le \omega/\eps^2\le\sigma_0/(\eps\omega)$ hold whp.
\end{lemma}
\begin{proof}
Set $k=\omega/\eps$ and $t_0'=\omega/\eps^2$.
It is easy to check that $k\le \sigma_0/\omega$ and $t_0'\le \sigma_0/(\eps\omega)$,
so it suffices to prove that $Z\le k$ and $T_0\le t_0'$ hold whp.

Note first that $t_0=\omega(n/\eps)^{1/2}$, so $t_0^{3/2}/n=\omega^{3/2}(\eps^3n)^{-1/4}=o(\omega^{3/2})=o(k)$.
Let $\cA_1$ denote the event that $M_{t_0}^* < k/8$, so $\cA_1$ holds whp by \eqref{M*}.
Let $\cA_2$ be the event that $M_{t_0'} < k/8$, and $\cA_3$ the event
that $M_t\le \eps t/4$ for all $t_0'\le t\le t_0$.
Then, noting that $k/\sqrt{t_0'}=\sqrt{\omega}\to\infty$,
the events $\cA_2$ and $\cA_3$ hold whp by Lemma~\ref{Smax}.

From Lemma~\ref{traj} we have $x_t=nx(t/n)\ge \eps t/2\ge 0$ for
all $t\le t_0=o(\eps n)$. Suppose that $\cA=\cA_1\cap \cA_2\cap \cA_3$ holds.
Then we have $x_t+S_t\ge x_t-M_t \ge \eps t/4-k/8$ for all $t\le t_0$,
using $\cA_2$ for $t\le t_0'$, and $\cA_3$ for $t>t_0'$.

For $t\le t_0$, from \eqref{Ei1a} and the fact that $t_0=o(n)$ we have
\[
 |E_t| \le M^*_{t_0}+ 2ctC_t/n
 \le k/8 +C_t
\]
if $\cA_1$ holds and $n$ is large.

At time $T_0$ we have $X_{T_0}=-2Z$ (see \eqref{ti}) and $C_{T_0}=Z$.
Suppose that $\cA$ holds. Then $\tX_{T_0}=X_{T_0}-E_t\le -Z+k/8$.
Since $\tX_{T_0}=x_{T_0}+S_{T_0}\ge \eps T_0/4-k/8$,
it follows that
\[
 -Z+k/8 \ge \tX_{T_0} \ge \eps T_0/4 -k/8.
\]
Rearranging gives $Z\le k/4-\eps T_0/4\le k/4$ and hence,
since $Z\ge 0$, $T_0\le k/\eps$, completing the proof.
\end{proof}

Let $T_1=\inf\{t:X_t=-2(Z+1)\}$, so $T_1$ is the first time after $t_0$ at which we finish
exploring a component. In particular, there is a component with $T_1-T_0$ vertices.
\begin{lemma}\label{T1T0}
$T_1-T_0$ is asymptotically normally distributed with mean $\rho n$
and variance $2\mu_1\eps^{-1}n$, where $\rho$ is defined by \eqref{rhodef}.
\end{lemma}
\begin{proof}
For $t\le T_1$ we have $C_t\le Z+1$, which is whp at most $\omega/\eps+1$ by Lemma~\ref{start}.
For $t\le \min\{T_1,\tmax\}$ it follows that $t C_t/n=O(\omega)$.
Since $\tmax^{3/2}/n=\Theta(\eps^{3/2}n^{1/2})$, the bounds \eqref{Ei1a}
and \eqref{M*} imply that whp
\begin{equation}\label{Ebd}
 \max_{t\le \min\{T_1,\tmax\}} |E_t| \le \omega \sqrt{\eps^3 n} \le \sigma_0/\sqrt{\omega},
\end{equation}
say.

Ignoring the irrelevant rounding to integers, let $t_1=\rho n$.
Let $t_1^-=t_1-t_0$ and $t_1^+=t_1+t_0$. Recalling \eqref{tmax},
we have $t_1^+\le\tmax=O(\eps n)=O(\sigma_0^2)$.
Hence, by Lemma~\ref{Smax}, $M_{t_1^+}=\Op(\sigma_0)$.
Since $X_t=x_t+S_t+E_t$ it follows that
\begin{equation}\label{close}
 \max_{t\le \min\{T_1,t_1^+\}} |X_t-x_t| \le \sqrt{\omega}\sigma_0
\end{equation}
holds whp.

Let $a=-\xdot(\rho)$, so from 
Lemma~\ref{traj},
\begin{equation}\label{eqa}
 a = -\xdot(\rho)  \sim \eps.
\end{equation}

Since $x(\rho)=0$ and $\ddot{x}$ is uniformly bounded,
recalling that $t_0=o(\eps n)$ 
it follows easily that $x_{t_1^-}$ and $x_{t_1^+}$ are
both of order $\eps t_0 = \omega\sigma_0$. To be concrete, if
$n$ is large enough, then we certainly have
\[
 x_{t_1^-} \ge 10\sqrt{\omega}\sigma_0 \hbox{\quad and\quad} x_{t_1^+}\le -10\sqrt{\omega}\sigma_0,
\]
say. By Lemma~\ref{traj} we have $x_{t_0}\ge \eps t_0/2\ge 10\sqrt{\omega}\sigma_0$.
Also $x_t$ increases near (within $o(\eps n)$ of) $t=0$, decreases near $t=t_1$,
and is of order $\Theta(\eps^2 n)$ in between. It follows that
$\inf_{t_0\le t\le t_1^-} x_t\ge 10\sqrt{\omega}\sigma_0$.
Let $\cB$ denote the event described in \eqref{close}.
Then, whenever $\cB$ holds, we have $X_t\ge 0$ for $t_0\le t\le \min\{T_1,t_1^-\}$.
Since $X_{T_1}\le -2(Z+1)<0$, and $T_1>t_0$ by definition, this implies $T_1> t_1^-$.

Recall from Lemma~\ref{start} that $Z\le \sigma_0$ whp. Suppose  $Z\le \sigma_0$,
$\cB$ holds, and $T_1>t_1^+$.
Then from $\cB$ and the bound on $x_{t_1^+}$ we have $X_{t_1^+}\le -9\sqrt{\omega}\sigma_0<-2Z-2$,
contradicting $T_1> t_1^+$. It follows that $T_1\le t_1^+$ holds whp.

We claim that
\begin{equation}\label{claim}
 \sup_{|t-t_1|\le t_0} |\tX_t-\tX_{t_1}-a(t_1-t)| = \op(\sigma_0).
\end{equation}
From \eqref{tXdef} we may write $\tX_t-\tX_{t_1}$ as
\[
 (x_t-x_{t_1}) + (S_t-S_{t_1}).
\]
Recalling that $t_1=\rho n$, $\xdot(\rho)=-a$ and that $\ddot{x}$ is uniformly bounded,
the difference
between the first term and $a(t_1-t)$ is $O(|t-t_1|^2/n)=O(t_0^2/n) = o(\sigma_0)$.
Since $(S_t-S_{t_1^-})_{t=t_1^-}^{t_1^+}$ is a martingale with final variance $O(t_0)=o(\sigma_0^2)$,
Doob's maximal inequality gives
$\sup_{|t-t_1|\le t_0}|S_t-S_{t_1}|=\op(\sigma_0)$, and \eqref{claim} follows.

Recall from Lemma~\ref{start} that $Z$, the number of components
explored by time $t_0$, satisfies $Z=\op(\sigma_0)$.
We have shown above that whp $T_1=\inf\{t:X_t=-2(Z+1)\}$ lies between $t_1^-$ and $t_1^+$.
From \eqref{Ebd}, $X_t$ is within $\op(\sigma_0)$ of $\tX_t$
at least until $T_1$.
It follows that at time $T_1$, we have $\tX_t=\op(\sigma_0)$.
Since $a=\Theta(\eps)$, \eqref{claim} thus gives
\begin{equation}\label{T1}
 T_1=t_1 +\tX_{t_1}/a +\op(\sigma_0/\eps).
\end{equation}
From Lemma~\ref{norm}, \eqref{tXdef} and the fact that $x(\rho)=0$, we have that $\tX_{t_1}$
is asymptotically normal with mean $0$ and variance $v_0 \rho n$.
Hence $\tX_{t_1}/a$ is asymptotically normal with mean 0 and variance
\[
 v_0 \rho n/a^2 \sim 2\mu_1\eps^{-1} n,
\]
using \eqref{vo1}, \eqref{rhosim} and \eqref{eqa}.
Since this variance is of order $\eps^{-1}n=\eps^{-2}\sigma_0^2$, the $\op(\sigma_0/\eps)$ error
term in \eqref{T1} is irrelevant, and $T_1$ is asymptotically normal
with mean $t_1=\rho n$ and variance $2\mu_1\eps^{-1}n$. Finally, from Lemma~\ref{start}
we have $T_0=\op(\sigma_0/\eps)$.
It follows that $T_1-T_0$ is asymptotically normal with the parameters claimed in the theorem.
\end{proof}

We are now ready to complete the proof of Theorem~\ref{th1}.

\begin{proof}[Proof of Theorem~\ref{th1}]
Let $\cC$ denote the component explored from time $T_0$ to $T_1$.
We have already shown that $\cC$ has the size claimed; two tasks
remain, namely to study the nullity of $\cC$, and to show that
there are no other `large' components.

Recall from \eqref{Ebeta} that the conditional expected
number of back-edges added at each step satisfies
\begin{equation}\label{ba}
 \E(\newbe_{t+1}\mid \cF_t) = \E(\eta_{t+1}-1\mid \cF_t)A_t/U_t +O(1/n).
\end{equation}
For the nullity, we shall consider only $t\le \min\{T_1,\tmax\}$,
recalling that whp $T_1\le \tmax=O(\eps n)$. In this range,
we have $C_t\le Z+1\le 2\omega/\eps$ whp by Lemma~\ref{start}.
Also $|E_t|\le\omega\sqrt{\eps^3 n}$ whp by~\eqref{Ebd}. Since $\eps\to 0$
and $\eps^3 n\to\infty$, if $\omega\to\infty$ sufficiently slowly
then both these bounds are $o(\sqrt{\eps n})$.
Recalling that $\tX_t=X_t-E_t=A_t-2C_t-E_t$ it follows that
whp
\begin{equation}\label{tXtAt}
 |\tX_t-A_t|=o(\sqrt{\eps n})
\end{equation}
throughout our range.

For $t\le \tmax$, Lemma~\ref{traj} gives $x_t=O(\eps^2 n)$,
and it follows easily from the bounds above that $\tX_t$ is whp $O(\eps^2 n)$.
Recalling that $f''(1)/f'(1)=\mu_2/\mu_1\sim 1$, and defining $z$ by $f(z)+t/n=1$ as before,
by Corollary~\ref{c2}
the maximum relative error (for $t\le \tmax$) in approximating $U_t$ by $nzf'(z)$ 
or $\E(\eta_{t+1}-1\mid \cF_t)$ by $zf''(z)/f'(z)$ is
$\Op(\sqrt{\eps n}/n)=\op(1/\sqrt{\eps^3 n})$. Using \eqref{ba} and \eqref{tXtAt}
it follows that whp
\begin{equation}\label{b1}
 D_{t+1}^* = \E(\newbe_{t+1}\mid \cF_t) = \frac{\tX_t}{n}\frac{f''(z)}{f'(z)^2} +o(\sqrt{\eps/n})
\end{equation}
for all $t\le T_1$.

Recalling that $x_{t_1}=0$,
the bounds in the proof of Lemma~\ref{T1T0} show that whp $|\tX_t|\le\omega \sigma_0$
for $t_1^-\le t\le T_1$, and that whp $T_1\le t_1^-+2t_0$.
It follows from \eqref{b1} that whp no more than, say, $\omega^2 \sigma_0t_0/n
=\omega^3\sigma_0^2/(\eps n)=\omega^3=o(\sqrt{\eps^3 n})$
back-edges are added between time $t_1^-$ and
time $T_1$. A similar bound holds for steps up to $t_0$ and for steps
between $T_1$ and $t_1=\rho n$.
Let $Y=Y_{\rho n}$ be the total number of back-edges found up to time $\rho n$.
Using \eqref{nullity}, it follows that
\begin{equation}\label{nC}
 |Y-n(\cC)| = \op(\sqrt{\eps^3 n}).
\end{equation}
Let us write $\newbe_{t+1}$ as $D_{t+1}^*+\Delta_{t+1}^*$, so by definition
$\E(\Delta_{t+1}^*\mid \cF_t)=0$. Then $Y=D^*+S^*$, where
$D^*=\sum_{t<\rho n} D_t^*$ and $S^*=\sum_{t<\rho n}\Delta_t^*$.
Note that $D^*$ is random. Recalling that $\tX_t=x_t+S_t$ and
that $\E S_t=0$, it follows from \eqref{b1} and the definition of $x_t=x(t/n)$ (see \eqref{xdef})
that
\[
 \E D^* = \sum_{t=0}^{\rho n-1} \left(zf'(z)-\frac{f'(z)^2}{\mu_1}\right) \frac{f''(z)}{f'(z)^2} +\op(\sqrt{\eps^3 n}),
\]
where, as usual, $z=z(t)$ is defined by $f(z)+t/n=1$.

It is not hard to see that the sum above sufficiently well approximated by the corresponding
integral. Recalling that with $\tau=t/n$ we have $\frac{\dd\tau}{\dd z}=-f'(z)$, it follows that
\[
 \E D^* = n\int_{z=z_1}^{1} (z-f'(z)/\mu_1)f''(z) \dd z +\op(\sqrt{\eps^3 n}),
\]
where $z_1$ corresponds to $\tau=\rho$. In other words, $z_1$ is the value
of $z$ defined in \eqref{prerho}, which as noted in the proof of Lemma~\ref{traj} satisfies $z_1=f'(z_1)/\mu_1$.
The integrand above is the derivative of $zf'(z)-f(z)-f'(z)^2/(2\mu_1)$.
It follows that
\begin{equation}\label{ED*}
 \E D^* = n\left(f(z_1)-\frac{\mu_1z_1^2}{2} -1 +\frac{\mu_1}{2}\right) + \op(\sqrt{\eps^3 n})
 = \rho^* n+\op(\sqrt{\eps^3 n}),
\end{equation}
recalling \eqref{rho*}. Since $\E\Delta^*=0$, this gives $\E Y=\rho^*n +\op(\sqrt{\eps^3n})$.

From Lemma~\ref{traj}\ref{v} and \eqref{dztau} it follows easily that
$\delta=1-z_1\sim \rho/\mu_1\sim 2\eps\mu_1/\mu_3$.
We may write $\rho^*$ as $f(1-\delta)-1+\mu_1\delta-\mu_1\delta^2/2$.
Expanding $f(z)$ about $z=1$, using
$f(1)=1$, $f'(1)=\mu_1$, $f''(1)=\mu_2=(1+\eps)\mu_1$, $f'''(1)=\mu_3$
and $f^{(4)}=O(1)$, a little calculation establishes that
\begin{equation}\label{rho*sim}
 \rho^* \sim  \frac{2\mu_1^3}{3\mu_3^2}\eps^3.
\end{equation}

We now turn to the variance and covariance estimates. Here we can be much less careful,
as a $o(1)$ relative error does not affect our conclusions.

Recall that $D^*$ is random. From \eqref{b1}, \eqref{ED*}
and the fact that $\tX_t=\E\tX_t+S_t$ we can write $D^*$
as $\rho^* n + D'+\op(\sqrt{\eps^3 n})$, where
\[
 D' = \sum_{t=0}^{\rho n-1} \frac{S_t}{n}\frac{f''(z)}{f'(z)^2}.
\]
Thus
\[
 Y = D^*+S^* = \rho^* n + D' +S^* + \op(\sqrt{\eps^3 n}).
\]
Throughout the relevant range, $f''(z)/f'(z)^2\sim f''(1)/f'(1)^2=\mu_2/\mu_1^2\sim 1/\mu_1$.
Since $S_t=\sum_{i<t} \Delta_i$, it follows that
\[
 D' = \sum_{i=0}^{\rho n-1} a_i \Delta_i,
\]
for some constants $a_i=a_i(n)$ satisfying
\[
 a_i\sim (\rho n-i)/(\mu_1 n).
\]
Since the $\Delta_i$ are martingale differences
with variances $v_0+o(1)$, it follows that
\[
  \Var(D') \sim \sum_{i=0}^{\rho n-1} v_0 a_i^2 \sim  \frac{v_0}{\mu_1^2} \sum_{i=0}^{\rho n-1}\frac{(\rho n-i)^2}{n^2}
 \sim \frac{v_0\rho^3}{3\mu_1^2} n \sim \frac{8\mu_1^3}{3\mu_3^2}\eps^3n,
\]
using \eqref{vo1} and \eqref{rhosim}.
Hence $D^*=\rho^*n +\Op(\sqrt{\eps^3 n}) = (1+\op(1))\rho^* n$.

Recalling that $\newbe_{t+1}\ge 2$ is much less likely that $\newbe_{t+1}=1$,
we have
\[
 \Var(\Delta^*_{t+1}\mid \cF_t) = \Var(\newbe_{t+1}\mid \cF_t) \sim \E(\newbe_{t+1}\mid \cF_t).
\]
Summing, it follows that
\[
 s^2=\sum_{t=0}^{\rho n-1} \Var(\Delta^*_{t+1}\mid \cF_t) =(1+\op(1)) D^* = (1+\op(1)) \rho^* n.
\]
In particular $s^2/(\rho^*n)$ converges in probability to $1$.

Recall that given $\cF_t$, $\Delta_{t+1}$ and $\Delta^*_{t+1}$ both have conditional
expectation $0$. Recalling \eqref{pform}, their conditional covariance,
which is just that of $\eta_{t+1}-1$ and $\newbe_{t+1}$, is asymptotically
$\Var(\eta_{t+1}-1\mid \cF_t)A_t/U_t\sim v_0A_t/U_t=\Theta(\eps^2)$.
This is much smaller than the square root of the product of their variances.
It follows easily that, after appropriate normalization, the joint distribution of
$S=S_{\rho n}=\sum_{t<\rho n}\Delta_t$
and $N'=D'+S^*=\sum_{t<\rho n} (a_t\Delta_t+\Delta_t^*)$ is asymptotically multi-variate normal, with
\[
 \Var(N') \sim \Var(D')+\Var(S^*) \sim \frac{10\mu_1^3}{3\mu_3^2}\eps^3n
\]
and
\[
 \Cov(N',S) \sim \Cov(D',S) \sim \sum_{t=0}^{\rho n-1} a_t v_0 \sim \frac{\rho^2 v_0}{2\mu_1}n \sim \frac{2\mu_1^2}{\mu_3}\eps^2 n.
\]
Recalling that the nullity of $\cC$ is $Y+\op(\sqrt{\eps^3 n})=\rho^*n+N'+\op(\sqrt{\eps^3 n})$
and that $|\cC| = \rho n+ S/a+\op(\sqrt{\eps n})$ where $a=-\xdot(\rho)\sim\eps$,
it follows that $|\cC|$ and $n(\cC)$ are jointly asymptotically normal with the means, variances
and covariance claimed in Theorem~\ref{th1}.

It remains only to prove that all  components other than $\cC$
have size bounded by $\Op\bb{\eps^{-2}\log(\eps^3 n)}=o(\eps n)$
as in \eqref{12}.
Lemma~\ref{start} shows that $T_0=\Op(\eps^{-2})$, so whp by time $T_1$
we have found no second component larger than this.

Let us stop the exploration at time $T_1$. Then the unexplored part of the graph
is simply the configuration multigraph on the degree sequence $\vd'$ given
by the vertices not yet reached.
Note that $\la(\vd')$ is exactly the expected value, given the history, of the degree $\eta_{T_1+1}$ 
of the vertex about to be chosen.
Since we have explored $O(\eps n)=o(n)$
vertices, $\vd'$ satisfies the assumption \eqref{a2} (with a slightly reduced $c_0$),
and it is still bounded. Since $\sqrt{\eps n}/n=o(\eps)$, by Corollary~\ref{c2}
$\E(\eta_{T_1+1}-1\mid T_1,\cF_{T_1})=\xdot(\rho)+\op(\eps)$,
so we find that $\la(\vd')=1-a+\op(\eps)$ with $a$ as in \eqref{eqa}. Since $a=\Theta(\eps)$,
Theorem~\ref{th-1} thus tells us that the largest component remaining
has size $\Op(\eps^{-2}\log(\eps^3 n))$, as required.
\end{proof}

\begin{remark}
Considering a random walk with independent increments with
distribution $\eta-2$, one would expect that the probability that our
random walk $(X_t)$ `takes off' without hitting $-2$ near the start is
roughly the expected degree of the initial vertex times $2\eps/v_0$.
(To see this, simply solve the recurrence relation for the
probability of hitting $-2$ as a function of the initial value.) The expected degree of the initial
vertex (which is chosen with probability proportional to degree) is
roughly $2$, giving a `take-off' probability of roughly $4\eps/v_0\sim
4\mu_1/\mu_3$; it is not hard to check that this
is asymptotically correct in the actual process $(X_t)$.

One should expect this probability to be simply related to (or at first
sight equal to) $\rho$; here the difference is that $\rho$ is the probability
that a \emph{uniformly} chosen random vertex is in the giant component.
It is not hard to check that the giant component is `tree-like', and
in particular has average degree $2+o(1)$, so the probability
that a vertex chosen with probability proportional to degree
is in the giant component is around $2\rho/\mu_1$.
From \eqref{rhosim} this is asymptotically $4\eps\mu_1/\mu_3$,
as it should be.

This comment illustrates a strange feature of the trajectory tracking arguments
here and in the papers of Nachmias and Peres~\cite{NP_regular} and Bollob\'as and Riordan~\cite{BR_walk}:
there is a related viewpoint using branching processes which more easily gives the approximate
size of the giant component, essentially by considering the probability that a vertex
is in a large component, i.e., that the random trajectory `takes off'. One can prove
quite accurate bounds by this method without worrying about when the trajectory will
hit zero eventually; see Bollobas and Riordan~\cite{rg_bp}. However,
for the distributional result, it seems easier to follow the whole trajectory.
\end{remark}

\section{A local limit theorem}\label{sec_local}

As a step towards the proof of Theorem~\ref{th-1} we shall need a
local limit theorem (Lemma~\ref{llt} below)
that may perhaps be known, but that we have
not managed to find in the literature. This concerns a sequence $(S_n)$
of sums of independent random variables. As usual in the combinatorial
setting, each $S_n$ involves different variables; we cannot
make the more usual assumption in probability theory that each $S_n$ is the sum
of the first $n$ terms of a single sequence. This makes little difference
to the proofs, however.

We start from Esseen's inequality
in the following form, also known as the Berry--Esseen Theorem;
see, for example, Petrov~\cite[Ch. V, Theorem 3]{Petrov}.
We write $\phi(x)$ and $\Phi(x)$ for the density and distribution
functions of the standard normal random variable.
\begin{theorem}\label{BE}
Let $Z_1,\ldots,Z_t$ be independent random variables such that $\rho=\sum_{i=1}^t \E(|X_i|^3)<\infty$,
and let $S=\sum_{i=1}^t Z_i$.
Then
\[
 \sup_x \bigabs{ \Pr( S\le x ) - \Phi((x-\mu)/\sigma) } \le A\rho/\sigma^3,
\]
where $\mu$ and $\sigma^2$ are the mean and variance of $S$,
and $A$ is an absolute constant.\noproof
\end{theorem}
Given an integer-valued random variable $Z$, let $b_r(Z)$ denote the $r$th \emph{Bernoulli part} of $Z$,
defined by
\begin{equation}\label{bp}
 b_r(Z) = 2\sup_i \min\{\Pr(Z=i),\Pr(Z=i+r)\}.
\end{equation}
It is easy to check that for any $p\le b_1(Z)$ we can write $Z$
in the form 
\begin{equation}\label{bdecomp}
 Z=Z'+IB,
\end{equation}
where $I\sim \Bern(p)$,
$B\sim\Bern(1/2)$, and $B$ is independent of the pair $(Z',I)$.
Here $\Bern(p)$ denotes the Bernoulli distribution assigning probability $p$
to the value $1$ and probability $1-p$ to the value $0$.
Similarly, for $p\le b_r(Z)$ we can write $Z$ in the form $Z'+rIB$ with the same
assumptions on $Z'$, $I$ and $B$.

We should like a `local limit theorem' giving, under mild conditions,
an asymptotic formula for $\Pr(S=\floor{\mu})$, say,
where $S$ is a sum of independent random variables and $\mu=\E S$. As is well
known (see, e.g., \cite[Ch. VII]{Petrov}), when the summands
take integer values in a finite range, the only obstruction
is their taking values in a non-trivial arithmetic progression,
in which case the sum \emph{cannot} take certain values.
Results similar to the next lemma are stated in~\cite{Petrov},
but the conditions are different in important ways.

\begin{lemma}\label{lllt}
Let $k\ge 1$ be fixed. Suppose that for each $n$ we have a sequence $(Z_{ni})_{i=1}^{t(n)}$
of independent random variables taking values in $\{-k,-k+1,\ldots,k\}$.
Let $S_n=\sum_{i=1}^{t(n)} Z_{ni}$, and let $\mu_n$ and $\sigma_n^2$ denote the mean
and variance of $S_n$.
Suppose that $\sigma_n^2=\Theta(t(n))$, that $t(n)\to\infty$, and that
\[
 b_{1,n} = \sum_{i=1}^{t(n)} b_1(Z_{ni}) \to\infty.
\]
Then for any sequence $(x_n)$ of integers satisfying $x_n=\mu_n+O(\sigma_n)$ we have
\begin{equation}\label{p0}
 \Pr(S_n=x_n) \sim p_0(n) = \frac{1}{\sigma_n\sqrt{2\pi}} \exp\left(-\frac{(x_n-\mu_n)^2}{2\sigma_n^2}\right).
\end{equation}
\end{lemma}
\begin{proof}
The condition $|Z_{ni}|\le k$ ensures that $\E(|Z_{ni}|^3)<k^3=O(1)$, while by assumption
$\sigma_n^2=\Theta(t(n))$, so
Theorem~\ref{BE} gives
\[
 \sup_x \bigabs{ \Pr(S_n\le x) -\Phi((x-\mu_n)/\sigma_n) } =O(t(n)^{-1/2}).
\]
We shall not use exactly this bound, instead applying Theorem~\ref{BE}
to a slightly different sum of independent variables.

Let $\omega=\omega(n)$ be an integer chosen
so that $\omega\to\infty$, but $\omega^6\le b_{1,n}$ and $\omega^{24}\le t(n)$, say.
Choose $p_i=p_{ni}$ so that $0\le p_i\le b_1(Z_{ni})$ and $\sum p_i=\omega^6$.

Suppressing the dependence on $n$ in the notation, let us write $Z_i=Z_{ni}$
in the form $Z_i=Z_i'+I_iB_i$ as in \eqref{bdecomp}, where $I_i\sim\Bern(p_i)$,
$B_i\sim \Bern(1/2)$, $B_i$ is independent of $(Z_i',I_i)$,
and variables associated to different $i$ are independent.
Let $\bI=(I_1,\ldots,I_t)$, and let $N=|\bI|=\sum I_i$.

The idea is simply to condition on $\bI$, and thus on $N$.
Let $S^0=S_n^0=\sum Z_i'$, and $S^1=S_n^1=\sum_{i:I_i=1} B_i$, so $S_n=S^0+S^1$. Then,
given $\bI$, $S^0$ and $S^1$ are independent. Furthermore,
the conditional distribution of $S^1$ is binomial $\Bi(N,1/2)$.

In the following argument we shall consider values of $N$ in three
separate ranges: $N\ge t^{1/3}$, $N\le \omega^6/2$, and the `typical'
range $\omega^6/2\le N\le t^{1/3}$.

Since $N$ is a sum of independent indicators with $\E N=\omega^6\le t^{1/4}$,
standard results (e.g., the Chernoff bounds) imply
that $\Pr(N\ge t^{1/3})$ is exponentially small in $t^{1/3}$,
and hence, extremely crudely, that $\Pr(N\ge t^{1/3})=o(t^{-1/2})$.
Thus
\begin{equation}\label{r1}
 \Pr\bb{S_n=x_n \wedge N\ge t^{1/3}} = o(t^{-1/2}).
\end{equation}

Note that $\E S^0=\mu_n-\E N/2=\mu_n-\omega^6/2=\mu_n+o(\sqrt{t})$.
Similarly, $\Var S^0= \Var S_n+o(\sqrt{t})=\sigma_n^2+o(\sqrt{t})$.
Let $\tilde\mu$ and $\tilde\sigma^2$ denote the conditional mean and variance
of $S^0$ given $\bI$. Given $\bI$, the summands $Z_i'$ in $S^0$
are independent, but their individual distributions depend on the $I_i$.
Changing one $I_i$ only affects the distribution of one
summand, and all $Z_i'$ are bounded by $\pm k$,
so we see that $\tilde\mu$ and $\tilde\sigma^2$ change by $O(1)$
if one entry of $\bI$ is changed.
It follows easily that $\tilde\mu=\E S^0+O(N+\E N)$ and $\tilde\sigma^2=\Var S^0+O(N+\E N)$.
Hence, when $N\le t^{1/3}$ we have
\begin{equation}\label{tms}
 \tilde\mu = \mu_n+o(t^{1/2}) \text{\ \ and \ \ }\tilde\sigma \sim \sigma_n.
\end{equation}
Given $\bI$, $S^0$ is a sum of $t$ independent random variables
whose third moments are all bounded by $k^3$.
By Theorem~\ref{BE} it follows that when $N\le t^{1/3}$ then
\begin{equation}\label{S0}
 \Pr(S^0< x\mid \bI) =\Phi((x-\tilde\mu)/\tilde\sigma) +O(t^{-1/2})
\end{equation}
uniformly in $x$. Considering consecutive (integer) values of $x$, it follows that
\begin{equation}\label{max}
 \sup_x\Pr(S^0=x\mid \bI)=O(t^{-1/2})
\end{equation}
whenever $N=|\bI|\le t^{1/3}$.

To handle the case $N\le \omega^6/2$, recall that after conditioning
on $\bI$, the sums $S^0$ and $S^1$ are independent. Thus
\begin{eqnarray*}
 \Pr(S_n=x_n \mid N\le \omega^6/2)
 &\le& \sup_{\bI_0: |\bI_0|\le \omega^6/2}\sup_m \Pr(S^0=x_n-m\mid \bI=\bI_0, S^1=m) \\
 &=&O(t^{-1/2}),
\end{eqnarray*}
using \eqref{max} for the final bound.
Since $\Pr(N\le \omega^6/2)=o(1)$,  this gives
\begin{equation}\label{r2}
 \Pr\bb{S_n=x_n \wedge N\le \omega^6/2} =o(t^{-1/2}).
\end{equation}

Finally, consider the `typical' case, where $\omega^6/2\le N\le t^{1/3}$.
Condition on $\bI$, assuming that $N$ is in this range.
Let $I$ be an `interval' consisting of $\omega$ consecutive integers. By \eqref{S0} we have
\[
 \Pr(S^0\in I\mid \bI) = \Phi(y+\omega/\tilde\sigma)-\Phi(y)+O(t^{-1/2}),
\]
where $y=(\min I-\tilde\mu)/\tilde\sigma$.
If the endpoints of $I$ are within $o(\sqrt{t})$ of $x_n=\mu_n+O(\sigma_n)$ then using
\eqref{tms} we have $y\sim(x_n-\mu_n)/\sigma_n$, and it follows easily that
\begin{equation}\label{I}
 \Pr(S^0\in I\mid \bI) = \omega p_0+O(t^{-1/2}) \sim \omega p_0,
\end{equation}
where $p_0$ is as in \eqref{p0}, and we used $p_0=\Theta(t^{-1/2})$
and $\omega\to\infty$ in the final approximation.

For $a=0,1,2,\ldots,$ let $J_a$ be the interval $[a\omega,(a+1)\omega-1]$,
and let $I_a=x_n-J_a$. Recalling that $S^0$ and $S^1$ are conditionally independent, 
we have
\[
 \Pr(S_n=x_n\mid \bI) \ge \sum_{a=0}^{N/\omega+1} \Pr(S^0\in I_a \mid\bI) \min_{j\in J_a}\Pr(S^1=j\mid \bI),
\]
and a corresponding upper bound with $\min$ replaced by $\max$.
Since $N\le t^{1/3}=o(\sqrt{t})$, from \eqref{I}
we have $\Pr(S^0\in I_a\mid\bI)\sim \omega p_0$ for all $a\le N/\omega+1$, so we obtain
\[
 \Pr(S_n=x_n\mid \bI) \ge (1+o(1)) \omega p_0 \sum_a \min_{j\in J_a}\Pr(S^1=j\mid \bI),
\]
and a corresponding upper bound with $\min$ replaced by $\max$.
Recall that, conditional on $\bI$, the distribution of $S^1$ is simply binomial $\Bi(N,1/2)$.
Since the standard deviation $\sqrt{N}/2$ of $S^1$ is much larger than $\omega$,
elementary properties of the binomial distribution imply that
\[
 \sum_a \min_{j\in J_a}\Pr(S^1=j\mid \bI) \sim  \sum_a \max_{j\in J_a}\Pr(S^1=j\mid \bI) \sim 1/\omega.
\]
(The bulk of each sum comes near the middle of the binomial distribution, where
the point probabilities hardly change within one interval; all that
is actually needed here is $\omega=o(\sqrt{N})$.)
This gives us an estimate for $\Pr(S_n=x_n\mid \bI)$ valid whenever $\omega^6/2\le N\le t^{1/3}$,
and it follows that
\begin{equation}\label{r3}
 \Pr\bb{S_n=x_n \wedge \omega^6/2< N <t^{1/3}} \sim p_0 \Pr(\omega^6/2< N <t^{1/3})\sim p_0.
\end{equation}
Combining \eqref{r1}, \eqref{r2} and \eqref{r3} gives the result.
\end{proof}

Results somewhat similar to Lemma~\ref{lllt} are certainly known; see, 
for example, McDonald~\cite{McDonald}, where Bernoulli parts are used
to deduce a local limit theorem from a central limit theorem. However,
the assumptions are different, and the Bernoulli part needed is much larger.

Note that uniform boundedness is not really needed in Lemma~\ref{lllt}; a suitable condition
on the third moments should suffice. Also, we may replace the condition
$\sum_i b_1(Z_{ni})\to \infty$ by $\sum_i b_r(Z_{ni})\to\infty$
for all $r\in R$, where $R$ is any set of integers with highest common factor $1$.
This latter condition is `almost' necessary (after passing to a subsequence): without it there is some
$d$ (the highest common factor of the integers in $R$) such that even distribution
modulo $d$ will only happen because of a `coincidence'; see the discussion
in \cite[Ch. VII]{Petrov}.

We shall need a result along the lines of Theorem~\ref{lllt} but away
from the central part of the distribution. This follows easily using
a trick called `exponential tilting', introduced by Cram\'er~\cite{Cramer},
and suggested to us by Paul Balister. Let $Z$ be a random variable, here
with finite support, such that $\Pr(Z>0)$ and $\Pr(Z<0)$ are both positive.
Consider the function
\[ 
 f(\alpha) = \E(Z e^{\alpha Z}) = \sum_x p_x x e^{\alpha x},
\]
where $p_x=\Pr(Z=x)$.
Note that
\[
 f'(\alpha) = \E(Z^2e^{\alpha Z}) >0.
\]
Also, if the support of $Z$ is contained in $[-k,k]$,
then $|f''(\alpha)| = |\E(Z^3e^{\alpha Z})| \le k^3e^{|\alpha| k}$.

Since $f(\alpha)$ is increasing and tends to $\pm\infty$ as $\alpha\to\pm\infty$,
there is a unique $a=a(Z)$ such that $f(a)=0$.
Define
\begin{equation}\label{cdef}
 c=c(Z)=\E(e^{aZ})=\sum_x p_x e^{ax},
\end{equation}
and let $Z'$ be the random variable with
\[
 \Pr(Z'=x) = \Pr(Z=x) e^{ax}/c,
\]
noting that these probabilities sum to 1 by the definition of $c$,
and that $\E(Z')=0$ by the definition of $a$. It is easy
to check that if $S_t$ denotes the sum of $t$ independent copies of $Z$
and $S_t'$ the sum of $t$ independent copies of $Z'$, then
\begin{equation}\label{tilt}
 \Pr(S_t'=x) = \Pr(S_t=x) e^{ax}/c^t.
\end{equation}

Recall that $b_1(Z)$ is the Bernoulli part of $Z$, defined by \eqref{bp}.

\begin{lemma}\label{llt}
Let $F$ be a finite set of integers, and let $Z_n$, $n\ge 1$, be a sequence
of probability distributions supported on $F$, with $\liminf\Pr(Z_n<0)>0$
and $\liminf\Pr(Z_n>0)>0$. Suppose that $t=t(n)\to\infty$, and that
$tb_1(Z_n)\to\infty$.
Let $S_n$ denote the sum of $t$ independent copies of $Z_n$,
and define $a_n=a(Z_n)$, $c_n=c(Z_n)$ and $Z_n'$ as above. Then
\[
 \Pr(S_n=x) \sim\frac{1}{\tilde\sigma_n\sqrt{2\pi t}} e^{-a_nx} c_n^t
\]
uniformly in integer $x=o(\sqrt{t})$, where $\tilde\sigma_n^2$ is the variance of $Z_n'$.
\end{lemma}
\begin{proof}
In the light of \eqref{tilt}, it suffices to prove that
$\Pr(S_n'=x)\sim 1/(\tilde\sigma_n\sqrt{2\pi t})$, where $S_n'$ is the sum of $t$
independent copies of $Z_n'$.

Passing to a subsequence, we may suppose that $\Pr(Z_n=i)$ converges for each $i$,
and that there are $i<0$ and $j>0$ for which the limit is strictly positive.
It follows that the `tilting amounts' $a=a(Z_n)$ are bounded. Hence
$b_1(Z_n')=\Theta(b_1(Z_n))$, so $tb_1(Z_n')\to\infty$.
Also, $\Var(Z_n')$ is bounded below by some positive number.
Lemma~\ref{lllt} thus applies to the sum of $t$ independent copies of $Z_n'$,
giving the result.
\end{proof}

We finish this section by noting some basic properties of tilting applied to random variables whose
mean is close to zero.
\begin{lemma}\label{smalltilt}
Let $k$ be fixed. If $Z_n$ is a sequence of distributions on $\{-k,\ldots,k\}$
with $\eps_n=\E Z_n\to 0$ and $\Var(Z_n)=\sigma_n^2=\Theta(1)$,
then the quantities $a_n=a(Z_n)$ and $c_n=c(Z_n)$ defined above satisfy
\begin{equation}\label{an}
 a_n \sim -\eps_n/\sigma_n^2
\end{equation}
and
\begin{equation}\label{cn}
 1-c_n \sim  \eps_n^2/(2\sigma_n^2).
\end{equation}
Furthermore, $\Var(Z_n')\sim \sigma_n^2$, and if $W_n$ is supported
on $\{-k,\ldots,k\}$ and may be coupled to agree with $Z_n$
with probability $1-p_n$ where $p_n=o(\eps_n)$, then
\begin{equation}\label{vary}
 |a(W_n)-a_n| = O(p_n)\text{\ \ and\ \ }|c(W_n)-c_n|=O(\eps_np_n).
\end{equation}
\end{lemma}
\begin{proof}
Suppressing the dependence on $n$, let $f(\alpha)=\E(Ze^{\alpha Z})$ as above.
Then $f(0)=\E Z=\eps$, and $f'(0)=\E Z^2=\sigma^2+\eps^2\sim\sigma^2$.
Also, $f''(\alpha)$ is uniformly bounded for $\alpha\in [-1,1]$, say.
It follows easily that $a_n$, the solution to $f(\alpha)=0$, satisfies \eqref{an}.
Similarly, let $g(\alpha)=\E(e^{\alpha Z})$. Then $g(0)=1$, $g'(0)=\eps$,
$g''(0)=f'(0)\sim \sigma^2$ and $g'''(\alpha)$ is bounded for $\alpha=O(1)$.
It follows that $c_n=g(a_n)=1+a_n\eps+a_n^2\sigma^2/2+O(\eps^3)$, giving \eqref{cn}.

To see that $\Var(Z_n')\sim\sigma_n^2$ it is enough to note that $a_n\to 0$.
The final part may be proved using the fact that for each fixed $j$ and all
$\alpha \in[-1,1]$, we have
$|\E(Z_n^j e^{-\alpha Z_n}) - \E(W_n^j e^{-\alpha W_n})|=O(p_n)$.
\end{proof}

\section{The subcritical case}\label{sec_below}

In this section we prove Theorem~\ref{th-1}.
Although we do use the exploration process considered
in the rest of the paper, we do not track the deviations of this
process from its expectation;
instead we use stochastic domination arguments to `sandwich'
the process between two processes with independent increments.

We start with a lemma concerning the tail of the distribution
of the time that a certain random walk with independent
increments takes to first hit a given value.
In the application we shall essentially take $Z_n$ to be the distribution of $\eta-2$,
where $\eta$ is the degree of a vertex of our graph chosen with probability
proportional to its degree (see Section~\ref{sec_explore}).
In fact, we shall adjust the distribution slightly
both to allow us to use stochastic domination, and
to meet the condition $t b_1(Z_n)\to\infty$. In what follows we often suppress dependence
on $n$ in the notation. Recall that the Bernoulli part $b_1(Z)$ of a distribution
$Z$ is defined by \eqref{bp}.

\begin{lemma}\label{Ltr}
Let $k\ge 1$ be fixed, and let $Z_n$ be a sequence of probability distributions
on $\{-1,0,1,\ldots,k\}$ converging in distribution to some $Z$
with $\Var(Z)>0$, such that $\eps=\eps(n)=-\E Z_n>0$ and $\eps\to 0$.
Let $\bW_n=(W_t)_{t\ge 0}$ be a random walk with $W_0=0$ and the
increments independent with distribution $Z_n$, and let $\tau_r=\tau_r(n)=\inf\{t:W_t=-r\}$.
Suppose that $r\ge 1$ is fixed, and that $t=t(n)$ is such that
$\eps^2 t\to\infty$ and $tb_1(Z_n)\to\infty$. Then
\begin{equation}\label{tr}
 \Pr(\tau_r\ge t) \sim c_{r,Z}\delta^{-1}t^{-3/2}e^{-\delta t},
\end{equation}
where $\delta=\delta_n=-\log(c(Z_n))$ with $c(Z_n)$ defined as in
\eqref{cdef},
and $c_{r,Z}>0$ is some constant depending on $r$ and $Z$.
\end{lemma}
\begin{proof}
We start with the case $r=1$.
Here Spitzer's Lemma~\cite{Spitzer} gives $\Pr(\tau_1=t)=\Pr(W_t=-1)/t$;
indeed, given a sequence of possible values $x_1,\ldots,x_t$ of the first $t$ increments summing
to $-1$, there is exactly one cyclic permutation of $(x_1,\ldots,x_t)$
such that the walk with the permuted increments stays non-negative
up to step $t$.

Lemma~\ref{smalltilt} gives $a_n\to 0$ and $c_n=1-\Theta(\eps^2)$.
Thus $\delta=\Theta(\eps^2)$ and, writing $\tilde\sigma^2$ for the variance of $Z_n'$,
$\tilde\sigma\sim \sigma$. Lemma~\ref{llt} thus gives
\[
 \Pr(W_t=-1) \sim \frac{1}{\sigma\sqrt{2\pi t}} e^{-\delta t},
\]
whenever $t\to\infty$ with $tb_1(Z_n)\to\infty$. Hence
\begin{equation}\label{t1=}
 \Pr(\tau_1 =t) \sim t^{-3/2}\frac{1}{\sigma\sqrt{2\pi}} e^{-\delta t}.
\end{equation}
When $\delta t\to\infty$, summing over $t$ easily gives
\begin{equation}\label{t1}
 \Pr(\tau_1\ge t) \sim \delta^{-1}t^{-3/2}\frac{1}{\sigma\sqrt{2\pi}} e^{-\delta t}.
\end{equation}
Indeed, the sum is dominated by the first $O(\delta^{-1})=O(\eps^{-2})$ terms,
and in this range $t^{-3/2}$ hardly changes.

For general $r$ we simply note that $\tau_r$ is distributed as the sum of $r$
independent copies $\tau^{(1)},\ldots,\tau^{(r)}$ of $\tau_1$. 
Since $\sum_{t\ge 1}t^{-3/2}$ converges, using \eqref{t1=} it is easy
to see that the dominant contribution to $\Pr(\tau_r\ge t)$
comes from the case that one of the $\tau^{(i)}$ is large and the others are $O(1)$.
Convergence in distribution of $Z_n$ implies that for each $j$ $\Pr(\tau_1=j)$ converges
to some limit, so \eqref{tr} follows from \eqref{t1}.
\end{proof}

We are now ready to prove Theorem~\ref{th-1}.

\begin{proof}[Proof of Theorem~\ref{th-1}]
Let $Z=Z_n$ denote the distribution of $\eta-2$, recalling that $\eta$
is the degree of a vertex chosen with probability proportional to its degree.
Passing to a subsequence, we may assume that $Z_n$ converges in distribution
to some distribution $Z^*$. Nonetheless, in what follows we must work
with the actual distribution $Z_n$ rather than the limit, since
the bounds are sensitive to small changes in the distribution of $Z_n$.

Note that $Z=Z_n$ is supported on $\{-1,0,\ldots,\dmax-2\}$, and that
by \eqref{vo1} we have $\Var(Z)=\Theta(1)$.
Also, $\E Z=-\eps$, where by assumption $\eps\to 0$ and $\eps^3 n\to\infty$.
Let $\delta=\delta(Z)=-\log(c(Z))$ be defined as above, noting that
\begin{equation}\label{dl}
 \delta \sim \frac{\eps^2}{2v_0} = \Theta(\eps^2)
\end{equation}
by Lemma~\ref{smalltilt}.
Note also that $\delta$ is exactly the quantity $\delta_n$ appearing in the statement
of Theorem~\ref{th-1}.

Let $\La=\eps^3n$, recalling that $\La\to\infty$ by assumption, and set
\[
 t^+ = \delta^{-1}(\log\La-\frac{5}{2}\log\log\La +\omega)
\]
for some $\omega\to\infty$ with $\omega=o(\log\log\La)$.

Let $Z'$ be defined in the same way as $Z$, except that we first remove
the $2\dmax t^+=o(n)$ non-isolated vertices of lowest degree from our degree sequence.
Then, conditional on the first $t\le 2t^+$ steps of our process,
using \eqref{Xdiff} the distribution
of the next increment $X_{t+1}-X_t=\eta_{t+1}-2-2\newbe_{t+1}\le \eta_{t+1}-2$
is stochastically dominated by $Z'$. Let $\gamma=\eps^{3/2}/\log\La$, say,
noting that $\gamma\to 0$, $\gamma t^+\sim\eps^{-1/2}\to\infty$, and $\eps\gamma t^+\to 0$.
Define $Z^+$ by modifying the distribution of $Z'$ as follows:
Pick some $k$ such that $\Pr(Z'=k)\ge 2\gamma$, and shift
mass $\gamma$ from $k$ to $k+1$. Note that $b_1(Z^+)\ge \gamma$,
so $t^+b_1(Z^+)\to\infty$.
Also $Z^+$ stochastically dominates $Z'$,
and $Z^+$ and $Z$ may be coupled to agree with probability 
$1-p$, where
\[
 p=O(t^+/n+\gamma).
\]
Note that $\eps^{-1}t^+/n=\Theta((\log\La)/\La)=o(1)$,
and clearly $\gamma=o(\eps)$, so $p=o(\eps)$.

Considering the random walk with independent increments distributed as $Z^+$,
writing $\cC_1$ for the first component revealed by our exploration,
stochastic domination and Lemma~\ref{Ltr} give
\[
 \Pr(|\cC_1|\ge t^+) = \Pr\bb{\inf\{t:X_t=-2\}\ge t^+}
  \le (1+o(1))c(\delta^+)^{-1}(t^+)^{-3/2} e^{-\delta^+t^+},
\]
where $\delta^+$ is defined as $\delta$, but using $Z^+$ in place of $Z$,
and $c=c_{2,Z^*}$ is a positive constant.

Lemma~\ref{smalltilt} gives $\delta^+\sim \delta$, 
and indeed $|\delta^+-\delta|=O(\eps p)$.
Thus $|\delta^+t^+-\delta t^+|=O(\eps p t^+)$, which is easily seen to be $o(1)$.
Thus the bound above can be written more simply as
\begin{eqnarray*}
 \Pr(|\cC_1|\ge t^+) &\le& (1+o(1)) c \delta^{-1}(t^+)^{-3/2} e^{-\delta t^+} \\
 &=& c \frac{t^+}{n} n\delta^{-1}(t^+)^{-5/2}  \La^{-1}(\log\La)^{5/2} e^{-\omega} \\
 &\sim& c \frac{t^+}{n} (2 v_0)^{-3/2} e^{-\omega} \\
 &=& \frac{t^+}{n} \Theta(e^{-\omega}) = o(t^+/n).
\end{eqnarray*}
If our graph $G_{\vd}^\mm$ contains a component of order at least $t^+$,
then the probability that we explore this component first is at least $(2t^+-2)/(\dmax n)=\Theta(t^+/n)$.
It follows that $\Pr(L_1\ge t^+)=o(1)$, proving the upper bound in \eqref{L1-}.

Turning to the lower bound, we use stochastic domination in the other direction.
This time we must account for back-edges. At a given step $t\le 2t^+$,
the (conditional, given the history) probability of forming a back-edge
is $O(t^+/n)$, simply because there can only be $O(t^+)$ active stubs.
It follows that we can define a distribution $Z^-$
that may be coupled to agree with $Z$ with probability $1-O(t^+/n+\gamma)$
so that the conditional distribution of $X_{t+1}-X_t$ stochastically
dominates $Z^-$ whenever $t\le 2t^+$. Setting
\begin{equation}\label{t-def}
  t^- = \delta^{-1}(\log\La-\frac{5}{2}\log\log\La -\omega),
\end{equation}
the argument above adapts easily to prove that
\[
 \Pr(|\cC_1|\ge t^-) \sim c\frac{t^-}{n}(2 v_0)^{-3/2}e^\omega.
\]
Let $I=[t^-,t^+]$ and write $t=(t^-+t^+)/2$, say.
Noting that $t^-\sim t^+\sim t$, the bounds above combine to give
$\Pr(|\cC_1|\in I)\sim c' e^{\omega}t/n$, where $c'=c(2 v_0)^{-3/2}$.
Let $N$ denote the number of components $\cC$ with $|\cC|\in I$.
It is easy to check that with high probability no such component
will have significantly more than $|\cC|$ edges.
Since our initial vertex is chosen with probability proportional to its degree,
it follows that
$\Pr(|\cC_1|\in I) \sim (\E N )2t/(\mu_1 n)$, where $\mu_1$ is the overall average degree.
Hence
\[
 \E N\sim c'\mu_1 e^{\omega}/2 \to\infty.
\]

Finally, with $\cC_2$ the second component explored by our process, we have
\begin{equation}\label{cc2}
 \Pr(|\cC_1|, |\cC_2|\in I) \sim \E(N(N-1)) (2t/(\mu_1 n))^2.
\end{equation}
The estimates above apply just as well to bound $\Pr(|\cC_2|\in I\mid |\cC_1|\in I)$:
throughout, we only needed that at most $2t^+$ vertices had been `used up'.
We find that the left-hand side in \eqref{cc2} is asymptotically $(c' e^\omega t/n)^2$,
so it follows that
\[
 \E(N(N-1)) \sim (c'\mu_1 e^{\omega}/2)^2 \sim (\E N)^2.
\]
Since $\E N\to\infty$, this gives $\E(N^2)\sim (\E N)^2$, so Chebyshev's inequality
implies that $\Pr(N>0)\to 1$, completing the proof of \eqref{L1-}.

The argument for \eqref{L1-+} is essentially the same; we simply replace $-\omega$ by
$+x$ in the definition \eqref{t-def} of $t^-$. With $N$ the number of components
with order between this new $t^-$ and $t^+$ it follows as above that $\E N\sim \alpha= c'\mu_1 e^{-x}/2$.
Moreover, arguing as for $N(N-1)$ above, for each fixed $r$ the $r$th factorial moment
of $N$ converges to $\alpha^r$. It follows by standard results that $N$ converges
in distribution to a Poisson distribution with mean $\alpha$, so $\Pr(N=0)\to e^{-\alpha}$.
Note that the constant
$c$ in \eqref{L1-+} is
\begin{equation}\label{cform}
 c= c'\mu_1/2=c_{2,Z^*}(2v_0)^{-3/2}\mu_1/2 \sim c_{2,Z^*} 2^{-5/2}\mu_1^{5/2}\mu_3^{-3/2},
\end{equation}
with $c_{2,Z^*}$ as in Lemma~\ref{Ltr}.
\end{proof}

\begin{ack}
The author would like to thank B\'ela Bollob\'as for many helpful discussions,
as well as for the invitation to visit the University of Memphis, and Paul
Balister for suggesting the use of `tilting' in the proof of Lemma~\ref{llt}.
\end{ack}

\end{document}